\documentclass[a4paper,11pt]{article}
\usepackage[left=1.8cm,right=1.8cm,top=2.5cm,bottom=2.5cm]{geometry}
\usepackage[english]{babel}
\usepackage{amssymb, amsfonts, amsmath, amsthm}
\usepackage{authblk}
\usepackage{epstopdf}
\usepackage{xcolor}
\usepackage{booktabs}
\usepackage{bm, bbm}
\usepackage{dsfont}
\usepackage{xr-hyper}
\usepackage{caption, url}
\usepackage{diffcoeff}
\usepackage{float}
\usepackage{algorithm}
\usepackage{algpseudocode}
\usepackage{enumitem}
\usepackage{graphicx}
\usepackage{subcaption}
\usepackage{mathtools}
\usepackage[
]{hyperref}
\usepackage[normalem]{ulem}
\usepackage{tikz}
\usetikzlibrary{arrows.meta}

\newcommand{\dd}{\, \mathrm{d}}

\newcommand{\nn}{\bm{n}}
\newcommand{\vs}{\bm{s}}
\newcommand{\uu}{\bm{u}}
\newcommand{\yy}{\bm{y}}
\newcommand{\UU}{\bm{U}}
\newcommand{\ww}{\bm{w}}

\newcommand{\cc}{\bm{c}}
\newcommand{\xx}{\bm{x}}

\newcommand{\Rset}{\mathbb{R}}

\newcommand{\Pprob}{\mathbb{P}}
\newcommand{\mcA}{\mathcal{A}}
\newcommand{\mcB}{\mathcal{B}}
\newcommand{\mcC}{\mathcal{C}}

\newcommand{\mcX}{\mathcal{X}}
\newcommand{\mcU}{\mathcal{U}}
\newcommand{\mcV}{\mathcal{V}}
\newcommand{\mcW}{\mathcal{W}}
\newcommand{\mcY}{\mathcal{Y}}
\newcommand{\mcP}{\mathcal{P}}
\newcommand{\Pd}{\mathcal{P}_d}

\DeclareMathOperator{\tr}{Tr}
\DeclareMathOperator{\var}{Var}
\newcommand{\cov}[2]{\text{Cov}\left(#1, #2\right)}
\newcommand{\ev}[1]{\mathbb{E}\left[#1\right]}
\newcommand{\cev}[2]{\mathbb{E}_{#2}\left[#1\right]}
\newcommand{\mycev}[2]{\mathbb{E}\left[#1\vert\, #2\right]}

\definecolor{tabblue}{RGB}{76,114,176}
\definecolor{taborange}{RGB}{221,132,82}

\newtheorem{definition}{Definition}[section]
\newtheorem{theorem}[definition]{Theorem}
\newtheorem{assumption}[definition]{Assumption}
\newtheorem{proposition}[definition]{Proposition}

\numberwithin{equation}{section}

\algrenewcommand\algorithmicrequire{\textbf{Input:}}
\makeatletter
\newenvironment{algorithmenv}
    {
        \par
        \addvspace{\medskipamount}
        \refstepcounter{algorithm}%
        \hrule height .8pt depth 0pt
        \kern 2pt
        \renewcommand{\caption}[2][\relax]
        {
            {\raggedright
             \textbf{\ALG@name~\thealgorithm} ##2\par}
            \ifx\relax##1\relax
                \addcontentsline{loa}{algorithm}
                  {\protect\numberline{\thealgorithm}##2}
            \else
                \addcontentsline{loa}{algorithm}
                  {\protect\numberline{\thealgorithm}##1}
            \fi
            \kern 2pt
            \hrule
            \kern 2pt
        }
    }
    {
        \kern 2pt
        \hrule
        \par
        \addvspace{\medskipamount}
    }
\makeatother

\newlength{\AlgAssignLHS}
\newcommand{\AlgAssign}[2]{
    \settowidth{\AlgAssignLHS}{#1 $\gets$\ }
    \State
        \makebox[\AlgAssignLHS][l]{#1 $\gets$\ }
        \parbox[t]{\dimexpr\linewidth-\AlgAssignLHS-2em\relax}
        {
            \raggedright
            #2
        }   
}

\graphicspath{ {./figures} }

\title{Global Sensitivity Analysis of Spatial Sets: Finite-Element-Based Estimation and the Role of Observation Windows}
\author[1]{Farbod Chamanian}
\author[1,2]{Chiara Piazzola}
\author[1]{Elisabeth Ullmann}
\affil[1]{Department of Mathematics, TUM School of Computation, Information and Technology, Technical University of Munich, Germany}
\affil[2]{Laboratori de Càlcul Numeric (LaCàN), ETS de Ingeniería de Caminos, Canales y Puertos, Universitat Politècnica de Catalunya - Barcelona Tech (UPC), Barcelona, Spain}
\date{}

\begin{document}

\maketitle

\section*{Abstract}
In this work we consider numerical models with random input parameters, where the model output is a spatially distributed random field.
We carry out a statistical sensitivity analysis of spatial sets which arise for instance when the model output exceeds a critical threshold. 
We consider two approaches: $(i)$ a kernel-based sensitivity analysis working with the Hilbert--Schmidt Independence Criterion, and $(ii)$ a function-valued sensitivity analysis working with generalized Sobol' indices, where the model output is the indicator function of the spatial set of interest.
We develop efficient estimators for the sensitivity indices, especially for finite-element-based numerical models, where we approximate volume integrals in the sensitivity measures by finite element quadrature, thereby avoiding Monte Carlo sampling over the spatial domain. 
This improves the state-of-the-art for the kernel-based indices and makes these methods practically accessible for expensive numerical models. 
We present numerical experiments for a hydrogen combustion process, modeled by a coupled system of convection-diffusion-reaction equations in a two-dimensional spatial domain, and ask whether the temperature remains below a critical value in selected regions of the combustion domain.
Our results show that in this example the kernel-based and the generalized Sobol' indices give qualitatively similar input importance rankings. Moreover, the ranking depends on the chosen observation window and can flip between different windows. The algorithms for implementing the proposed methods are provided in the supplementary materials and in an open source code repository.

\bigskip

\noindent
\textbf{Keywords:} uncertainty quantification, Sobol' indices, HSIC-ANOVA indices, partial differential equations with random coefficients, random sets

\bigskip

\noindent
\textbf{AMS Subject Classification:} 35-04, 
62-04, 
62P30 

\section{Introduction}
Across scientific fields, models serve as representations of systems, processes, and observed phenomena that embody our understanding and assumptions about their behavior.
However, models inevitably are approximations, constrained by the assumptions and simplifications used to construct them. Uncertainty quantification (UQ) aims to determine appropriate uncertainty measures associated with model-based predictions by modeling unknowns as random objects equipped with a suitable probability distribution \cite{handbookOfUQ}. 
Sensitivity Analysis (SA) is a branch of UQ which assesses the influence of each model input on the model output \cite{SA_book}.
Local Sensitivity Analysis (LSA) examines the effect of perturbing a specific input around a fixed value, typically using partial derivatives of the input-output map. 
Global Sensitivity Analysis (GSA) evaluates the influence of input parameters over their entire domain  using sensitivity measures which are based on the variance or other distributional information of the model output \cite{SA_book}.

In this work we focus on models involving partial differential equations (PDEs) with spatial, temporal, or spatiotemporal output functions.
Our leading example is a hydrogen combustion process, modeled by a coupled system of convection-diffusion-reaction PDEs in a two-dimensional domain with five random input parameters and four output random fields \cite{paper:CDR_paper2, paper:CDR_paper1}.
We are interested to learn whether the temperature remains below a critical value in selected spatial sub-regions of the combustion domain. 
Specifically, we consider the questions: \textit{How influential are the random inputs relative to each other and in absolute terms, in determining whether the temperature in a sub-region of the combustion domain remains below a critical value? 
Does the importance ranking of the inputs change across the spatial domain?}

We follow two approaches to study the critical-value-exceedance: In the direct approach, the model output is a random set and we carry out SA for set-valued outputs as introduced by Fellmann et al. in \cite{fellmann_kernelSA}. Alternatively, we consider function-valued outputs as in Gamboa et al. \cite{gamboa_SA_for_multiDim_and_funtional_outputs}, where the output is the indicator function associated with the random set considered in the set-valued, direct approach.
In the set-valued approach, the concept of a derivative of a set-valued map is not easily defined. 
In the function-valued approach, the input-output map is not differentiable.
Hence we employ derivative-free GSA techniques.  

Sensitivity analysis has typically been carried out for models with scalar-valued outputs \cite{de2008_uncertaintyInIndustrialPractice_aGuide, handbookOfUQ, iooss2015_GSA_in_UncertaintyManagementBook}.
However, for PDE-based models the outputs are often function-valued which, upon discretization, give rise to vector-valued outputs.
Set-valued outputs are also possible, such as in the hydrogen combustion model discussed above.
This brings novel challenges: SA methods for functional and set-valued outputs are considerably less developed, and the theoretical and computational complexity increases significantly compared to scalar-valued outputs.
The following selected works demonstrate this clearly.
Marrel et al. \cite{Marrel2010, Marrel2014} compute spatial maps of Sobol' indices for spatial and spatio-temporal model outputs, respectively.
De Lozzo et al. \cite{DeLozzo2017} use variance-based and kernel-based sensitivity measures for spatio-temporal model outputs.
Alexanderian et al. \cite{timeDependent_VBSA} carry out a variance-based SA for time-dependent models using the generalized Sobol' indices in \cite{gamboa_SA_for_multiDim_and_funtional_outputs}. 
Perrin et al. \cite{Perrin2021} propose a functional principal component analysis for spatial outputs. Fellmann et al. \cite{fellmann_kernelSA} introduce a set-kernel and carry out a kernel-based SA for set-valued model ouputs. 

If one starts with no assumptions on the model and allows for nonlinearity and non-monotonocity, de Rocquigny et al. \cite{de2008_uncertaintyInIndustrialPractice_aGuide} suggest the use of Sobol' indices as the sensitivity measure. 
Sobol' indices \cite{origSobolPaper_1, origSobolPaper_2} are well-known sensitivity measures based on the Sobol'--Hoeffding or functional analysis-of-variance (ANOVA) decomposition \cite{SA_book, Hoeffding:1948, origSobolPaper_1} for statistically independent model inputs and scalar-valued model outputs.
Gamboa et al. \cite{gamboa2013sensitivity_multivariateSobols} define Sobol' indices for models with multivariate outputs.
In our leading problem  we analyze the sensitivity of the temperature field in selected spatial regions.
To be able to use the Sobol' indices for multivariate outputs, one could discretize the spatial region by a grid and define the model output as a vector which collects the temperature evaluated at the grid points.
Instead, we work directly with functional outputs and employ the generalized Sobol' indices introduced by Gamboa et al. \cite[Sec.~6]{gamboa_SA_for_multiDim_and_funtional_outputs}.
Notably, the Sobol' indices in \cite{gamboa2013sensitivity_multivariateSobols} are defined in terms of the trace of covariance matrices associated with the model output.
Analogously, the generalized Sobol' indices in \cite{gamboa_SA_for_multiDim_and_funtional_outputs} are defined in terms of the trace of suitable covariance operators associated with the model output.
Thus, both of these Sobol' indices neglect the spatial correlations of the model output. 
Moreover, the typical pick-freeze estimators for the Sobol' indices \cite{gamboa_pickFreeze, janon2014asymptotic_of_aggr_and_scal_PickFreeze} are expensive with a cost that grows linearly in the number of model inputs.
Therefore, we also investigate the use of kernel-based sensitivity indices, specifically the HSIC-ANOVA indices  \cite{daVeiga2021kernel, SA_book}, which employ the Hilbert--Schmidt Independence Criterion (HSIC) \cite{Gretton2005_CrossCovIntro}.
These indices can handle set-valued outputs directly \cite{fellmann_kernelSA} and are relatively inexpensive to estimate with a cost that does not depend on the number of model inputs \cite{daVeiga2021kernel, fellmann_kernelSA}. 
We employ the generalized Sobol' indices as baseline for the comparison with the HSIC-ANOVA indices.
We note that the generalized Sobol' and the HSIC-ANOVA indices admit an ANOVA-like decomposition and are thus suitable for the ranking of the model inputs.
Moreover, both indices do not exclusively use the variance or other moments of the process model output.
Indeed, in our setting the generalized Sobol' indices use (pointwise in the spatial domain) the full distribution information of the process model because the process model is composed with the indicator function of the random set. 

Kernel-based sensitivity indices, see e.g. \cite{Barr:2022, daVeiga2015global, durrande2013anova, 2026TotalHSICwithDependentInputs}, employ the rich theory of Reproducing Kernel Hilbert Spaces (RKHSs) to develop sensitivity measures.
For example, \cite{daVeiga2015global, 2026TotalHSICwithDependentInputs} embed probability measures into an RKHS and use the HSIC.
Similarly to Sobol' indices, kernel-based indices have initially been developed for scalar-valued model outputs. 
However, a major advantage of the kernel-based approach is that the probability measures can be defined on measurable spaces without requiring a vector space structure, which is what allows one to work directly with set-valued model outputs.
Recently, Fellmann et al.\ \cite{fellmann_kernelSA} introduced a kernel for Lebesgue-measurable sets, motivated by the sensitivity analysis of spatial sets, and adapted the kernel-based indices in \cite{daVeiga2015global} to set-valued outputs. 
The numerical evaluation of the set kernel is done by employing a Monte Carlo sampling scheme in the spatial region of interest.  

{In this work, we focus on PDE-based models discretized using the finite element method (FEM).
This naturally provides the finite element (FE) basis for the spatial function approximation and allows efficient spatial integration using the FE mass matrix. 
Building on this framework, in Subsection~\ref{section:HSIC_estimation} we introduce an improved method to evaluate the set kernel used in the HSIC-ANOVA index. 
The same FEM framework is used in Subsection~\ref{section:SpIn_estimation} to estimate functional Sobol' indices associated with indicator functions of random sets. We introduce these indices in Section~\ref{section:SpIn} and refer to them as spatially-integrated (SpIn) Sobol indices.
Equipped with the two efficient estimators, we then study the aforementioned SA question on the temperature field of the hydrogen combustion problem.
Our results in Subsection~\ref{subsection:numerical_results} show that the choice of the observation window matters in the combustion process, as it can flip the importance rankings of the random inputs. 
All the algorithms are presented in the supplementary materials. Note that open-source code is also available.}

\subsection{Organization}
This article is structured as follows. 
In Section~\ref{section:problem_description} we formulate the problem, the relevant input-output maps, and assumptions. Section~\ref{section:kernel_based} recalls the kernel-based sensitivity indices using the HSIC. Section~\ref{section:SpIn} discusses the generalized Sobol' indices for function-valued outputs, in particular, indicator functions of spatial sets.
Section~\ref{section:estimation_of_sensitivity_indices} presents computable estimators for the indices in this work with emphasis on FEM-based numerical models.  Section~\ref{section:cdr_results} presents the hydrogen combustion model and its discretization, and numerical experiments, including a qualitative comparison of the indices and a study with different observation windows in the spatial domain. 
Section~\ref{section:conclusion} offers concluding remarks. 

\subsection{Notation}
For a positive integer $d$, let $1:d$ denote the set $\{1,2,\dots,d\}$.
Let $\Pd$ denote the set of all subsets of $1:d$.
For a subset $A\in \Pd$, let $A^c$ denote the set $1:d\backslash A$.
For an integer $j \in 1:d$, we use $j$ to represent the set $\{j\}$ whenever the context is clear.
Moreover, let $-j$ denote the set $1:d\backslash \{j\}$.
Let $|A|$ denote the cardinality of $A$. 
For ${\bm u} \in \mathbb{R}^d$, let $\uu_A$ denote the $|A|$-dimensional vector containing only the elements $u_j$, where $j \in A$. 
For $A=\emptyset$ we define $\uu_\emptyset=\bm{0}$.
For the random vector $\UU$ taking values in $\mathbb{R}^d$, let $\UU_A$ denote the random vector containing only the components $U_j$, where $j \in A$.
For $A=\emptyset$ we define $\UU_\emptyset:=0$.
The empty set $\emptyset$ is denoted by $0$ whenever the context is clear.
The Euclidean vector norm is denoted by $\Vert \cdot \Vert$.

\section{Problem description}\label{section:problem_description}

Let $(\Omega, \mathcal{F}, \Pprob)$ be a probability space and consider a random vector $\UU =(U_1, \ldots, U_{n_U})^\top$ with realizations in $ \mathbb{R}^{n_U}${, $n_U \in \mathbb{N}$, } whose components are independent random variables. 
We define 
$
\mathcal{U}
=
\bigtimes_{i=1}^{n_U}\mathcal{U}_i
\subseteq
\mathbb{R}^{n_U}
$ such that 
\begin{align*}
\UU:\Omega \longrightarrow \mathcal{U},  \quad \omega \longmapsto \uu:=\UU(\omega).
\end{align*}
We denote the associated probability distribution by $\Pprob_{\UU}$ for which, due to independence, it holds $\Pprob_{\UU}=\bigotimes_{i=1}^{n_U}\Pprob_{U_i}$, where $\Pprob_{U_i}$ is the probability distribution of each random variable $U_i$.
{Further, let $\mcX \subset \Rset^{n_x}$, $n_x \in \mathbb{N}$ be an open and bounded set with Lipschitz boundary,} for example, the spatial domain associated with a PDE-based model.
We consider the following process model
\begin{equation}\label{equation:model_output}
Y = G({\xx},  {\UU}), \quad {\xx} \in \mcX \subset \Rset^{n_x},
\end{equation}
defined via the deterministic, measurable map 
\begin{equation}\label{equation:G}
G:\mcX \times \mcU\rightarrow \mcY\subset\mathbb R.
\end{equation} 
The model output $Y$ is thus a random field \cite{Adler2010, Christakos1992}
\begin{equation}\label{equation:YRF}
Y:\mcX\times\Omega\rightarrow\mcY,\quad({\xx},\omega)\mapsto G({\xx},{\UU}(\omega)),
\end{equation}
whose randomness is induced by the random input vector $\UU$. 

We are interested in subsets $\mcX_\text{obs}\subseteq \mcX$ for which the model output satisfies a prescribed constraint. 
The set $\mcX_\text{obs}$ is called \textit{observation window}.
Let $\mcA\subseteq\mathcal{Y}$ 
denote the set of admissible outputs. 
For every realization $\uu\in\mathcal{U}$, we define the corresponding \textit{feasible set}
\begin{equation}\label{equation:feasible_set}
\Gamma(\uu)
=
\{\xx\in\mathcal{X}_\text{obs}\colon G(\xx,\uu)\in\mcA\}
\end{equation}
which can be viewed as a realization of a random set \cite{Molchanov2017}.  
Unless otherwise stated we use $\mcX_\text{obs}=\mcX$ in the following.
We return to the study of observation windows in Section~\ref{section:cdr_results}.
To formalize the random set contruction, we proceed as in \cite{fellmann_kernelSA}. 
Let $\mathcal{L}(\mathcal{X})$ denote the space of all Lebesgue-measurable subsets of $\mathcal{X}$, and let $\lambda$ be the Lebesgue measure on $\mathcal{X}$. Let 
$
\delta(\Gamma_1,\Gamma_2)
=
\lambda(\Gamma_1\Delta\Gamma_2)
$ 
be the Lebesgue measure of the symmetric difference between $\Gamma_1$ and $\Gamma_2$, i.e.~$ \Gamma_1\Delta\Gamma_2 = (\Gamma_1\cup\Gamma_2)\setminus(\Gamma_1\cap\Gamma_2)$. 
$\delta$ is a pseudometric on $\mathcal{L}(\mcX)$ and it holds 
    \begin{align}\label{eq:lebesgue_sym_diff}
        \delta(\Gamma_1,\Gamma_2)= \int_\mcX\mathds{1}_{\Gamma_1\Delta \Gamma_2} \dd\lambda = \int_\mcX(\mathds{1}_{\Gamma_1}-\mathds{1}_{\Gamma_2})^2 \dd\lambda 
        = \int_\mcX |\mathds{1}_{\Gamma_1}-\mathds{1}_{\Gamma_2}| \dd\lambda
    \end{align} 
    with the identity in the second equality sign being proved in \cite[Lemma 3.1]{fellmann_kernelSA}.
Quotienting $\mathcal{L}(\mcX)$ by the equivalence relation $\Gamma_1 \sim \Gamma_2$ iff $\delta(\Gamma_1,\Gamma_2)=0 $ yields the separable metric space $\mathcal{L}^*(\mcX)$, 
equipped with the Borel $\sigma$-algebra $\mathcal{B}_{\Gamma} = \mathcal{B}(\mathcal{L}^*(\mcX),\delta)$.
{In the following we slightly abuse the terminology and refer to the equivalence classes in the quotient space $\mathcal{L}^\ast(\mcX)$ as sets.}
The feasible-set construction therefore defines the set-valued map
\begin{equation}\label{eta}
\eta:\mathcal{U}\rightarrow\mathcal{L}^*(\mathcal{X}),
\qquad
\uu\mapsto \Gamma := \eta(\uu).
\end{equation}
Consequently, provided that $\eta$ is a measurable map from $(\mcU,\mathcal{B}_{\mcU})$ to $(\mathcal{L}^*(\mathcal{X}),\mathcal{B}_{\Gamma})$, where $\mathcal{B}_{\mcU}$ is the Borel $\sigma$-algebra on $\mcU$, the random input vector $\UU$ induces the random set $\Gamma(\UU) = \Gamma(\UU(\omega))$.
The probability distribution of $\Gamma$ is denoted by $\Pprob_{\Gamma}$ and is the push-forward of $\Pprob_{\UU}$ by $\eta$. We denote by $\Pprob_{\UU,\Gamma}$ the joint distribution of $\UU$ and $\Gamma$.

The aim of this work is to quantify the intensity of dependence between a process model's input parameters, collected in the random vector $\UU$, and the corresponding set-valued outputs $\Gamma(\UU)$. 
In Section \ref{section:kernel_based} we discuss the so-called HSIC-ANOVA indices \cite{daVeiga2021kernel, fellmann_kernelSA} which are kernel-based and can handle set-valued outputs. 
In Section \ref{section:SpIn} we consider generalized Sobol' indices for functional outputs \cite{gamboa_SA_for_multiDim_and_funtional_outputs}. We bridge to the set-valued outputs by assigning to a random set $\Gamma(\UU)$ the indicator random field 
\begin{equation}\label{indicatorRF}
\mathds{1}_{\Gamma}\colon \mcX \times \Omega \rightarrow \mathbb{R}, \quad (\xx,\omega) \mapsto \mathds{1}_{\{G(\xx,\UU(\omega))\in \mcA\}},
\end{equation}
whose randomness is again induced by $\UU$.
We consider $\mathds{1}_\Gamma \in L^2(\Omega; L^2(\mcX))$, the Bochner space of strongly measurable, $L^2(\mcX)$-valued random variables, where $L^2(\mcX)$ is equipped with the spatial $L^2$ norm.
To formalize the construction we define the function-valued map
\begin{equation}\label{theta}
\theta\colon \mcU \rightarrow L^2(\mcX), \qquad \uu\mapsto \mathds{1}_\Gamma := \theta(\uu).
\end{equation}
Provided that $\theta$ is a measurable map from $(\mcU,\mathcal{B}_{\mcU})$ to $(L^2(\mcX),\mathcal{B}(L^2(\mcX))$, where $\mathcal{B}(L^2(\mcX))$ is the Borel $\sigma$-algebra generated by all open subsets in $L^2(\mcX)$, the random input vector $\UU$ induces the random function $\mathds{1}_{\Gamma(\UU)} = \mathds{1}_{\Gamma(\UU(\omega))}$.
The probability distribution of $\mathds{1}_\Gamma$ is denoted by $\Pprob_{\mathds{1}_\Gamma}$ and is the push-forward of $\Pprob_{\UU}$ by the map $\theta$. 

{Thus, the sensitivity indices in this work are based on the two different maps identified above which can be derived from the process model in \eqref{equation:G}.
The HSIC-ANOVA indices in Section~\ref{section:kernel_based} are defined for the set-valued map $\uu \mapsto \Gamma(\uu)$ in \eqref{eta}.
The generalized Sobol' indices in Section~\ref{section:SpIn} are defined for the function-valued map $\uu \mapsto \mathbbm{1}_{\Gamma(\uu)}$ in \eqref{theta}. 
Under suitable assumptions both types of sensitivity indices are based on 
ANOVA-like decompositions, yet they differ in how they measure the input parameters' influence.
The kernel-based indices in Section~\ref{section:kernel_based} investigate kernel mean embeddings of the measures $\Pprob_{\UU_A,\Gamma}$ and $\Pprob_{\UU_A} \otimes \Pprob_{\Gamma}$ which are defined on the Cartesian product of a subset of the input space \textit{and} the output space of the map $\eta$ in \eqref{eta}.
They can be estimated with samples of the joint input-output distribution {without} conditioning.
In contrast, the generalized Sobol' indices in Section~\ref{section:SpIn} are based on the (conditional) measures $\Pprob_{\mathds{1}_{\Gamma}}$ and $\Pprob_{\mathds{1}_{\Gamma}\vert \UU_A}$ which are defined on the output space of the map $\theta$ in \eqref{theta}.
The conditioning may increase the estimation cost (see Section~\ref{section:estimation_of_sensitivity_indices}).
}

{Finally, we define a compact subset of $\mathcal{L}^\ast(\mcX)$ and formulate assumptions on the measure $\Pprob_{\Gamma}$ in the set-valued approach.
The assumptions are motivated by Mercer's Theorem (see e.g. \cite[Theorem 2.1]{Muandet2017}) and are used in Section~\ref{section:kernel_based} to derive the ANOVA-like decomposition of the kernel-based indices, see Proposition~\ref{proposition:hsic_decomposition}.}
{Let $c>0$ denote a constant. Let $\mathcal{L}_c(\mcX) \subset \mathcal{L}(\mcX)$ denote the space of all Lebesgue-measurable subsets of $\mcX$ with finite perimeter which is uniformly upper bounded by $c$.
Then $\mathcal{L}^{\ast}_c(\mcX)$ is defined as the quotient space of $\mathcal{L}_c(\mcX)$ with respect to the equivalence relation $\Gamma_1 \sim \Gamma_2$ iff $\delta(\Gamma_1, \Gamma_2)=0$ for $\Gamma_1, \Gamma_2 \in \mathcal{L}_c(\mcX)$. We can now formulate the following assumption.}
\begin{assumption}\label{ass:set}
{The measure $\Pprob_\Gamma$ is supported on $\mathcal{L}^\ast_c(\mcX)$.}
\end{assumption}
{Assumption~\ref{ass:set} formalizes the requirement that the realizations of the random set $\Gamma(\UU)$ have finite, uniformly bounded perimeters.
That $\mathcal{L}^\ast_c(\mcX)$ is compact follows from classical arguments in Analysis: 
We assign to a sequence $(S_k)_{k \in \mathbb{N}}$ with elements in $\mathcal{L}^\ast_c(\mcX)$ the sequence of indicator functions $(\mathbbm{1}_{S_k})_{k \in \mathbb{N}}$.
By assumption every $S_k$ has a finite perimeter which is equivalent to $\mathbbm{1}_{S_k} \in {BV}(\mcX)$, where ${BV}(\mcX)$ denotes the space of functions on $\mcX$ with bounded variation (BV) \cite[Chap.~5]{EvansGariepy2015}. Moreover, the BV norm of all functions $(\mathbbm{1}_{S_k})_{k \in \mathbb{N}}$ is uniformly bounded. By \cite[Theorem 5.5]{EvansGariepy2015} there exists an $L^1$-convergent subsequence of $(\mathbbm{1}_{S_k})_{k \in \mathbb{N}}$ with a limit element in $BV(\mcX)$.
This limit element is again an indicator function $\mathbbm{1}_S$ and the associated set $S$ has a finite perimeter. Since the perimeter is a lower semicontinuous function with respect to $L^1$ convergence \cite[Theorem 5.2]{EvansGariepy2015}, the perimeter of $S$ is upper bounded by $c$. Finally, the $L^1$ convergence of the indicator functions is equivalent to the convergence of the associated sets with respect to the symmetric difference measure $\delta$. Thus, $\mathcal{L}^\ast_c(\mcX)$ is sequentially compact which is equivalent to compactness in a metric space.
}

\section{Kernel-based sensitivity indices} \label{section:kernel_based}

In the following, we recall the theoretical foundations of kernel-based sensitivity indices for set-valued outputs. Their construction is based on an ANOVA decomposition of the Hilbert–Schmidt independence criterion, a measure of statistical dependence between random quantities. The formulation relies on the theory of Reproducing Kernel Hilbert Spaces,     kernel mean embeddings, and cross-covariance operators that we recall in Section \ref{subsection:theory_of_kernels} before giving the definition of the sensitivity indices in Section \ref{subsection:kernel_based_sensitivity_index}.

\subsection{Mathematical preliminaries} \label{subsection:theory_of_kernels}

Consider the probability space $(\mcV,\mathcal{B}_{\mcV},\Pprob_{\mcV})$, where $\mcV$ is a separable {metric} space, $\mathcal{B}_{\mcV}$ the corresponding Borel $\sigma$-algebra, and $\Pprob_{\mcV}$ a probability measure. 
A function $k:\mcV\times\mcV\rightarrow\mathbb{R}$ is called a \textit{kernel} if it is symmetric and positive definite.
Kernels allow one to embed elements of $\mcV$ into a Hilbert space of functions, the \textit{Reproducing Kernel Hilbert Space}.

\begin{definition}[Reproducing Kernel Hilbert Space]\label{definition:rkhs} 
    Let $\mathcal{H}$ be a Hilbert space of real-valued functions on $\mcV$. $\mathcal{H}$ is a \textit{Reproducing Kernel Hilbert Space (RKHS)} if there exists a kernel $k:\mcV\times\mcV\rightarrow\mathbb{R}$ such that, for all $v\in\mcV$, 
    \begin{enumerate} [label=\roman*.]
        \item $k(v,\cdot)\in\mathcal{H}$;
        \item $f(v)=\langle{f,k(v,\cdot)}\rangle_{\mathcal{H}}$ for all $f\in\mathcal{H}$.
    \end{enumerate} 
    The kernel $k$ is called the \textit{reproducing kernel} of $\mathcal{H}$.
\end{definition}
Every RKHS admits a unique reproducing kernel. Conversely, every kernel determines uniquely (up to an isomorphism) a RKHS for which it is the reproducing kernel \cite{kernel_RKHS_uniqueness}. 
From now on, we denote by $\mathcal{H}_k$ the RKHS associated to the reproducing kernel $k$.

This construction can be generalized to product spaces $\mcV = \bigtimes_{i=1}^p\mcV_i $, where $\mcV_i$, $i=1,\ldots,p$ are separable {metric} spaces. 
Let $\mathcal{H}_{k_i}$, $i=1,\ldots,p$ be the corresponding RKHS with reproducing kernels $k_i:\mcV_i\times\mcV_i\rightarrow\mathbb{R}$. 
The product kernel on $\mcV \times \mcV$ is defined by 
    \begin{equation}\label{eq:product_kernel}
    k \left( (v_{1}^{(1)},\ldots, v_{p}^{(1)}), \, (v_{1}^{(2)},\ldots,v_{p}^{(2)}) \right):= \prod_{i=1}^pk_i(v_i^{(1)},v_i^{(2)}), \quad v^{(1)}, v^{(2)} \in \mcV.
    \end{equation}
This kernel uniquely determines an RKHS $\mathcal{H}_k\coloneqq\bigotimes_{i=1}^p\mathcal{H}_{k_i}$ with $k$ as its reproducing kernel, see \cite{k_generates_RKHS_book_BerlinetAgnan2003}. 
The tensor product $\bigotimes_{i=1}^p\mathcal{H}_{k_i}$ is understood as the completion of the algebraic tensor product endowed with the canonical inner product induced by the factor Hilbert spaces. 

To extend the RKHS framework from deterministic elements to probability measures, we introduce the \emph{kernel mean embedding}, which represents probability distributions as elements of an RKHS; see, e.g., \cite{k_generates_RKHS_book_BerlinetAgnan2003,smola2007_MMDIntro}.

\begin{definition}[Kernel Mean Embedding]
    Let $M^1_+(\mcV)$ be the space of probability measures on $\mcV$.
    The \textit{Kernel Mean Embedding} (KME) of $M^1_+(\mathcal{V})$ into the RKHS $\mathcal{H}_k$ is defined by the mapping
    \begin{align} \label{eqn:kbsa_theory_basicTheory_kernelMeanEmbedding}
         \mu:M^1_+(\mathcal{V}) \rightarrow \mathcal{H}_k, \quad \mathbb{P} \mapsto \mu(\Pprob) = \int_{\mathcal{V}}{k(v,\cdot) \dd\mathbb{P}(v)} = :  \mu_{\mathbb{P}}.
    \end{align}
\end{definition}
The integral is to be interpreted as a Bochner integral, i.e.~an integral for Banach space-valued functions. 
Under the assumption that the kernel $k$ is bounded and continuous, the map $v \mapsto k(v,\cdot)$ is Bochner integrable, and hence the KME is well defined. 
Note that the KME uniquely characterizes a probability measure if $\mu$ is injective. This leads to the notion of a \textit{characteristic kernel}.
\begin{definition}[Characteristic Kernel] \label{definition:characteristic_kernels}  
    A kernel $k$ is said to be characteristic if $\mu_{\Pprob} = \mu_{\mathbb{Q}}$ $\Rightarrow$ $\Pprob = \mathbb{Q}$, for all $\Pprob, \mathbb{Q} \in M^1_+(\mathcal{V})$.
\end{definition}
{It is in general difficult to determine if a given kernel is characteristic.
We refer to \cite{Szabo2018Characteristic, Ziegel2024} for overviews.
In addition, the product of characteristic kernels is not always characteristic on the associated product space, see \cite{Szabo2018Characteristic}.}
We summarize the required properties for a given kernel $k$ in the following assumption.
\begin{assumption}\label{hp:kernel}
     {The kernel $k$ is bounded and continuous.}
\end{assumption}

While the KME represents a single probability distribution as an element of a RKHS, in this work we are interested in characterizing the dependence between two random objects, the model inputs and outputs. To this aim we consider a second probability space $(\mcW, \mathcal{B}_{\mcW}, \Pprob_{\mcW})$ with $\mcW$ a separable {metric} space, and introduce the \emph{cross-covariance operator}, see e.g.~\cite{baker1973CrossCovariance}. 

\begin{definition}[Cross-covariance operator] \label{definition:cross_covariance_operator}
Let $k_{\mcV}$ and $k_{\mcW}$ be {continuous and bounded} reproducing kernels on the separable {metric} spaces $\mcV$ and $\mcW$, respectively. Let $\mathcal{H}_{k_{\mcV}}$ and $\mathcal{H}_{k_{\mcW}}$ be the associated RKHSs and $\mathcal{H}_{k_{\mcV}}\otimes\mathcal{H}_{k_{\mcW}}$ the RKHS induced by product kernel $k_{\mcV,\mcW}$ over $\mcV\times\mcW$. Let $\mathbb{P}_{\mcV}\in M_+^1(\mcV)$ and $\mathbb{P}_{\mcW}\in M_+^1(\mcW)$ be probability measures on $\mcV$ and $\mcW$, respectively, and $\Pprob_{\mcV,\mcW} \in M_+^1(\mcV \times \mcW)$ a probability measure on $\mcV \times \mcW$. The (centered) cross-covariance operator 
$\mathcal C_{\mathbb P_{\mcV,\mcW}}: \mathcal H_{k_{\mathcal W}} \rightarrow \mathcal H_{k_{\mathcal V}}$
is defined by
\begin{align*}
\mathcal C_{\mathbb P_{\mcV,\mcW}}
&=
\int_{\mathcal V\times\mathcal W}
\left(k_{\mathcal V}(v,\cdot)-\mu_{\mathbb P_{\mathcal V}}\right)
\otimes
\left(k_{\mathcal W}(w,\cdot)-\mu_{\mathbb P_{\mathcal W}}\right)
\dd \Pprob_{\mcV,\mcW}(v,w)
\\
&=
\int_{\mathcal V\times\mathcal W}
k_{\mathcal V}(v,\cdot)\otimes
k_{\mathcal W}(w,\cdot)
\dd \Pprob_{\mcV,\mcW}(v,w)
-
\mu_{\mathbb P_{\mathcal V}}
\otimes
\mu_{\mathbb P_{\mathcal W}}.
\end{align*}
\end{definition}
Under Assumption \ref{hp:kernel} on the kernels {$k_\mcV$ and $k_\mcW$}, the cross-covariance operator {$\mathcal C_{\mathbb P_{\mcV,\mcW}}$} is a Hilbert--Schmidt operator and therefore has finite Hilbert--Schmidt norm; see, e.g., \cite{Gretton2005_CrossCovIntro}. We denote such norm by $\lVert \cdot \rVert_{HS}$.
This leads to the definition of the \emph{Hilbert--Schmidt Independence Criterion}, which quantifies the dependence between the probability measures $\mathbb P_{\mcV,\mcW}$ and $\mathbb P_{\mcV}\otimes\mathbb P_{\mcW}$, see \cite{Gretton2005_CrossCovIntro}.
\begin{definition} [Hilbert-Schmidt Independence Criterion] \label{definition:hsic}
    Let the setup of Definition \ref{definition:cross_covariance_operator} hold. The Hilbert-Schmidt Independence Criterion (HSIC) associated with the joint probability measure $\mathbb P_{\mcV,\mcW}$ is defined by 
    \begin{align} \label{equation:hsic_definition}
        \text{HSIC}(\mathbb P_{\mcV,\mcW}) \coloneqq& \lVert \mathcal{C}_{\mathbb P_{\mcV,\mcW}} \rVert^2_{HS}. 
    \end{align}
\end{definition}

With the characteristic property of the kernels, we present the following important theorem for the sensitivity analysis, see \cite[Theorem 4]{Gretton2005_CrossCovIntro} or \cite[Theorem 3(i)]{Szabo2018Characteristic}.
\begin{theorem} [Inferring Independence from HSIC]\label{thm:HSIC:independence}
    Let the setup of Definition \ref{definition:cross_covariance_operator} hold.
    {Further, let the kernels $k_\mcV$ and $k_\mcW$ be characteristic.}
    Then, 
    \begin{equation}
        \text{HSIC}(\Pprob_{\mcV,\mcW})=0 \iff \mathbb{P}_{\mcV,\mcW}=\mathbb{P}_{\mcV}\otimes\mathbb{P}_{\mcW}. 
    \end{equation}
\end{theorem}

\subsection{Kernel-based Sensitivity Index} \label{subsection:kernel_based_sensitivity_index}

The RKHS framework recalled in the previous section, and in particular the HSIC, can be used to define kernel-based sensitivity indices, see \cite{daVeiga2021kernel} {and \cite[Chap.~6]{SA_book}}. 
While the HSIC was introduced above in terms of a generic joint probability measure, we now specify it to the setting in Section \ref{section:problem_description} and consider the model inputs and output of the map $\eta$ in \eqref{eta} directly. Accordingly, for input variables $\UU_A$ and the set-valued output $\Gamma$, we use the shorthand
\[
\text{HSIC}(\UU_A,\Gamma) := \text{HSIC}(\Pprob_{\UU_A,\Gamma}).
\]

In analogy with classical Sobol' sensitivity indices, which admit an ANOVA decomposition of the variance, one can develop an ANOVA-like decomposition for HSIC. 
The key building blocks for this construction are input kernels (that is, kernels on $\mcU \times \mcU$) of ANOVA type, see \cite{durrande2013anova,ginsbourger2016anova} {and Mercer's Theorem for the output kernel, see e.g. \cite[Theorem 2.1]{Muandet2017}.} 
This enables the separation of contributions associated with different subsets of input variables.
{Note that the output kernel on $\mathcal{L}^\ast(\mcX) \times \mathcal{L}^\ast(\mcX) $ is not required to be an ANOVA kernel.}
\begin{definition}[ANOVA kernel]
\label{definition:anova_kernels}
A kernel $k_i:\mcU_i\times\mcU_i\rightarrow\mathbb{R}$, $i = 1,\ldots, n_U$  is called an ANOVA kernel with respect to $\Pprob_{U_i}$ if it can be decomposed as
\[
k_i = 1+\bar{k}_i,
\]
where $\bar{k}_i:\mcU_i\times\mcU_i\rightarrow\mathbb{R}$ is centered with respect to $\Pprob_{U_i}$, i.e.,
\begin{equation}\label{anova:centered}
\int_{\mcU_i} \bar{k}_i(u_i,u_i') \dd\Pprob_{U_i}(u_i) =0, \qquad \forall u_i'\in\mcU_i.
\end{equation}
The corresponding ANOVA kernel on $\mcU=\bigtimes_{i=1}^{n_U}\mcU_i$ is defined as the product kernel
\begin{equation} \label{equation:ANOVA_kernel_definition}
k_{\mathrm{ANOVA}}(\uu,\uu') = \prod_{i=1}^{n_U} \left(1+\bar{k}_i(u_i,u_i')\right), \qquad \uu,\uu'\in\mcU.
\end{equation}
\end{definition}
Note that expanding the product in \eqref{equation:ANOVA_kernel_definition} yields
\begin{equation}\label{anova:plus one}
k_{\mathrm{ANOVA}}(\bm{u},\bm{u}') = 1+ \sum_{\emptyset\neq A\subseteq \{1,\ldots,n_U\}} \prod_{i\in A}\bar{k}_i(u_i,u_i').
\end{equation}
Hence, the kernel decomposes into components associated with all subsets of the input variables.
Provided that the input random variables $U_i$, $i =1,\ldots, n_U$ are mutually independent as assumed in Section \ref{section:problem_description}, this structure induces a corresponding ANOVA decomposition of the HSIC, see \cite{daVeiga2021kernel} {and \cite[Chap.~6]{SA_book}}. 
In the presence of dependent inputs, we refer to \cite{daVeiga2021kernel} for kernel-based Shapley effects and to \cite{2026TotalHSICwithDependentInputs} for HSIC indices.

\begin{proposition} [ANOVA decomposition of HSIC; see {{\cite[Theorem 6.7]{SA_book}} and \cite{daVeiga2021kernel}}] \label{proposition:hsic_decomposition}
    For each $A\in\mathcal{P}_{n_U}$, define $\mcU_A:=\bigtimes_{i \in A} \mcU_i$ and let the reproducing kernel on $\mcU_A \times \mcU_A$ be given as $k_A\coloneqq\prod_{i\in A}k_i$. Let $k_{\Gamma}$ be the reproducing kernel of the output defined on $\mathcal{L}^*(\mcX)\times\mathcal{L}^*(\mcX)$. Let $k_A$ and $k_{\Gamma}$ satisfy Assumption \ref{hp:kernel}. Moreover, let $k_A$ be an ANOVA kernel 
    {and let the measure $\Pprob_\Gamma$ satisfy Assumption~\ref{ass:set}.}
    Then, the ANOVA decomposition of the HSIC {associated with $\Pprob_{\UU,\Gamma}$} is given by
    \begin{align}
        \text{HSIC}(\UU,\Gamma) = \sum_{A\in\mcP_{n_U}\setminus\{ \emptyset \}}\text{HSIC}(\UU_A,\Gamma), \label{equation:hsic_decomposition_full} 
    \end{align}
    where 
    \begin{equation}
        \text{HSIC}(\UU_A,\Gamma) =\sum_{B\subseteq{A}}(-1)^{|A|-|B|}\text{HSIC}(\UU_B,\Gamma). \label{equation:hsic_decomposition_subset_A}
    \end{equation}
\end{proposition}
The ANOVA structure of the input kernel leads to simplifications in estimating both $\text{HSIC}(\UU_A,\Gamma)$ in \eqref{equation:hsic_decomposition_subset_A} and $\text{HSIC}(\UU,\Gamma)$ in \eqref{equation:hsic_decomposition_full}, see Subsection~\ref{section:HSIC_estimation}.
We conclude this subsection with the presentation of the kernel-based HSIC-ANOVA sensitivity indices; see \cite{daVeiga2021kernel} {and \cite[Chap.~6]{SA_book}}.
\begin{definition} [HSIC-ANOVA Sensitivity Index] \label{definition:hsic_anova_index}
    Let the setup of Proposition \ref{proposition:hsic_decomposition} hold. The HSIC-ANOVA Sensitivity Index associated to a subset $A \in \mcP_{n_U}$ of the random input variables, is defined as
    \begin{equation}\label{equation:hsic_anova_index}
        S_A^{HSIC}\coloneqq\frac{\text{HSIC}(\UU_A,\Gamma)}{\text{HSIC}(\UU,\Gamma)}. 
    \end{equation} 
    Moreover, the following identity holds
    \begin{equation}
        \sum_{A\in\mathcal{P}_{n_U}\setminus\{\emptyset\}}S^{HSIC}_A=1,
    \end{equation}
    where we set $S_A^{HSIC}=0$ for $A=\emptyset$.
\end{definition}
In this paper, we study first order (main effect) HSIC-ANOVA indices. That is, we focus on $S_A^{HSIC}$, where $|A|=1$.
{Some remarks on the HSIC-ANOVA indices are in order. First, Theorem~\ref{thm:HSIC:independence} tells us that, provided both the input kernel $k_i$ associated with $U_i$ and the output kernel $k_\Gamma$ are characteristic, we can use the HSIC for the \textit{screening} of input variables: If $\text{HSIC}(U_i,\Gamma)=0$, then the random variable $U_i$ is statistically independent of the output $\Gamma(\UU)$. 
Second, the ANOVA decomposition of the HSIC in Proposition~\ref{proposition:hsic_decomposition} allows us to properly normalize the HSIC indices. 
Thus, they can be used for the \textit{ranking} of the input variables: 
the first order index $S_i^\text{HSIC}$ tells us how much of $\text{HSIC}(\UU,\Gamma)$ is explained by the random input $U_i$ alone, analogously to the main effect Sobol' index for real-valued outputs.} 

The estimation of the HSIC-ANOVA sensitivity indices is discussed in Section~\ref{section:HSIC_estimation}. 

\section{Spatially-integrated Sobol' sensitivity indices}\label{section:SpIn}

In this section we perform sensitivity analysis of indicator functions of spatial sets. 
We employ the framework of generalized Sobol' indices for functional outputs defined in \cite{gamboa_SA_for_multiDim_and_funtional_outputs}.
We recall in Section~\ref{section:SpIn:prelim} the theoretical foundations of function space valued random variables, their covariance operator and a trace formula which is crucial for the calculation of the generalized Sobol' indices.
In Section~\ref{section:SpIn:definition} we recall the definition of the generalized Sobol' index and present expressions for the indices which are practically computable.

\subsection{Mathematical Preliminaries}\label{section:SpIn:prelim}
Let $\mcV$ be a separable Hilbert space with inner product $\langle \cdot,\cdot\rangle_\mcV$ and induced norm $\Vert \cdot \Vert_\mcV$.
Let $(\Omega, \mathcal{F},\mathbb{P})$ be a probability space.
Then {$L^2(\Omega, \mathcal{F}; \mcV) = L^2(\Omega; \mcV)$} is the Bochner space of all $\mcV$-valued, strongly $\mathcal{F}$-$\mathcal{B}(\mcV)$-measurable random variables $Z\colon \Omega \rightarrow \mcV$, such that $\ev{\Vert Z\Vert^2_\mcV}<\infty$, where $\mathcal{B}(\mcV)$ is the Borel-$\sigma$ algebra generated by all open subsets in $\mcV$.
$L^2(\Omega;\mcV)$ is a Hilbert space with inner product $\ev{\langle \cdot, \cdot\rangle_\mcV}$ and norm ${\lVert \cdot \rVert_{L^2(\Omega;\mcV)}:=}\left(\ev{\Vert \cdot \Vert^2_\mcV}\right)^{1/2}$.
Elements of $L^2(\Omega; \mcV)$ are called mean-square (m.-s.) integrable random variables, see \cite[Chap.~4]{introToComputationalStochasticPDEs}

Let $Z \in L^2(\Omega; \mcV)$.
The mean of $Z$ is the unique element $\mu_Z \in \mcV$, such that $\ev{\langle v, Z \rangle_\mcV}=\langle v, \mu_Z\rangle_\mcV$ for all $v \in \mcV$.
The covariance operator $\mcC_Z$ associated with $Z$ is defined as $\mcC_Z \colon \mcV \rightarrow \mcV$, such that $\langle \mcC_Z v, w\rangle_\mcV = \cov{\langle Z, v\rangle_\mcV}{\langle Z, w\rangle_\mcV}$ for all $v,w \in \mcV$.
By definition it holds $\ev{\Vert {Z} \Vert^2_\mcV}=\Vert \mu_Z \Vert^2_\mcV<\infty$,
thus $\mu_Z \in L^2(\Omega; \mcV)$.
Moreover, $\mcC_Z$ is a self-adjoint, positive semi-definite, trace-class operator with trace $\tr{\mcC_Z}=\ev{\Vert Z-\mu_Z\Vert^2_\mcV}$, see e.g. \cite[Lemma 4.37]{introToComputationalStochasticPDEs}.

We are interested in $\mcV=L^2(\mcX)$.
In this case $Z(\xx)$ is a real-valued random variable for each $\xx \in \mcX$, and the mean function $\mu_Z(\xx)=\ev{Z(\xx)}$ and covariance function $c_Z(\xx,\xx'):=\cov{Z(\xx)}{Z(\xx')}$ are well defined a.s. in $\mcX$, see \cite[Chap.~7]{introToComputationalStochasticPDEs}.
Moreover, if $c_Z \in L^2(\mcX \times \mcX)$, then for $v,w \in L^2(\mcX)$, from the definition of $\mcC_Z$ and Fubini's Theorem we obtain
\begin{equation*}
\begin{split}
\langle \mcC_Z v, w\rangle_{L^2(\mcX)}
&=
\ev{\langle Z-\mu_Z, v \rangle_{L^2(\mcX)}\langle Z-\mu_Z, w \rangle_{L^2(\mcX)}}\\
&=
\ev{\int_\mcX (Z(\xx)-\mu_Z(\xx)) v(\xx)\dd \xx \int_\mcX (Z(\xx')-\mu_Z(\xx')) w(\xx')\dd \xx' }\\
&=
\int_\mcX \int_\mcX \cov{Z(\xx)}{Z(\xx')} v(\xx) w(\xx') \dd \xx \dd \xx',
\end{split}
\end{equation*}
so that the covariance operator takes the form of the integral operator
\begin{equation}\label{covariance integral operator}
(\mcC_Z v)(\xx)
= \int_\mcX \cov{Z(\xx)}{Z(\xx')} v(\xx') \dd \xx',\qquad v \in L^2(\mcX), \quad {x \in \mcX,}
\end{equation}
with kernel function $c_Z$.
For the trace of $\mcC_Z$, using Tonelli's Theorem, we obtain
\begin{equation}\label{trace is integrated variance}
\begin{split}
\tr{\mcC_Z}
&=
\ev{\Vert Z-\mu_Z\Vert^2_{L^2(\mcX)}}
=
\ev{\int_\mcX (Z(\xx)-\mu_Z(\xx))^2\dd \xx}\\
&=
\int_\mcX \ev{(Z(\xx)-\mu_Z(\xx))^2}\dd \xx
=
\int_\mcX \var Z(\xx)\dd \xx.
\end{split}
\end{equation}
This trace formula is the key in the evaluation of the generalized Sobol' indices.

We remark that \cite[Prop.~2.1]{timeDependent_VBSA} proves an analogous result for time-dependent processes under the assumption that the covariance function $c_Z$ is continuous on $\mcX \times \mcX$.
{In our setting described in Section \ref{section:problem_description}, $Z:=\mathds{1}_{\Gamma}$ is an indicator random field. Hence, $c_Z(\xx,\xx')$ can be written as sum of products of probabilities of certain events, and it is therefore bounded on $\mcX \times \mcX$. Consequently, it is square-integrable on the bounded domain $\mcX \times \mcX$, although it need not be continuous.} 

A further remark is in order. The covariance operator of a real-valued random field $Z$ is often constructed as the integral operator in \eqref{covariance integral operator} with kernel $c_Z$, see e.g. \cite{Ghanem1991, introToComputationalStochasticPDEs}.
If the kernel $c_Z \in L^2(\mcX \times \mcX)$, then the integral operator $\mcC_Z$ in \eqref{covariance integral operator} is a Hilbert--Schmidt operator on $L^2(\mcX)$ \cite[Thm.~1.65]{introToComputationalStochasticPDEs}.
However, a Hilbert--Schmidt operator is not necessarily trace class.
A classical argument to show $\tr \mcC_Z < \infty$ along with the trace formula in \eqref{trace is integrated variance} is by Mercer's Theorem under the assumption that the covariance function $c_Z$ is continuous, see e.g. \cite{timeDependent_VBSA}. 
In contrast, the definition of $\mcC_Z$ as covariance operator of a function space valued random variable $Z\in L^2(\Omega; \mcV)$ establishes $\tr \mcC_Z =\ev{\Vert Z - \mu_Z\Vert^2_\mcV} <\infty$ from elementary definitions. 
However, to be able to conclude $\tr \mcC_Z <\infty$ from elementary definitions, it must be proved that $Z$ is indeed a $\mcV$-valued random variable. 
In our setting $Z=\mathbbm{1}_\Gamma$ is defined as composition of the random vector $\UU$ and the map $\theta$ in \eqref{theta}.
Provided that $\theta$ is measurable, we can conclude that $\mathbbm{1}_\Gamma$ is a random variable, since it is the composition of a random variable and a measurable map.
Now, the measurability of the map $\theta$ in \eqref{theta} 
must be established on a case-by-case basis.
This requires further assumptions on the constraint set $\mcA$, the observation window $\mcX_\text{obs}$, and the process model $G$ in \eqref{equation:G} which are out of the scope of this work.

\subsection{Spatially-integrated Sobol' Index}\label{section:SpIn:definition}
{Building on the mathematical concepts recalled in the previous section, we introduce the generalized Sobol' indices (see \cite{gamboa_SA_for_multiDim_and_funtional_outputs}) defined in terms of the trace of covariance operators associated with the indicator random field $Z:=\mathds{1}_{\Gamma}$, thus addressing the setting described in Section \ref{section:problem_description}. }

Recall that $\UU$ is a $n_{U}$-valued random vector and $\Gamma=\Gamma(\UU)$ is a random set with realizations in $\mcX \subset \mathbb{R}^{n_{\xx}}$.
Note that by assumption $\lambda(\mcX)<\infty$, thus the Lebesgue measure of all realizations of $\Gamma$ is bounded.
Moreover, the realizations of $Z$ are in $L^2(\mcX)$ $\Pprob$-a.s. and $Z \in L^2(\Omega;L^2(\mcX))$.
Indeed, for $\omega \in \Omega$ we have
\[
\Vert Z(\cdot,\omega) \Vert^2_{L^2(\mcX)}
=
\int_\mcX \mathds{1}^2_{\Gamma}(\xx,\omega)\dd \xx
=
\lambda(\Gamma(\UU(\omega))<\infty
\]
and thus $\Vert Z \Vert^2_{L^2(\Omega;L^2(\mcX))}=\ev{\Vert Z(\cdot,\omega)\Vert^2_{L^2(\mcX)}}<\infty$.

Next, we recall an ANOVA-like decomposition of $Z=\mathds{1}_\Gamma$, see \cite{timeDependent_VBSA, gamboa_SA_for_multiDim_and_funtional_outputs} for general m.-s. integrable random variables.
The mean function of $Z$ is $Z_0(\xx):=\ev{\mathds{1}_{\{G(\xx,\bm{U})\in \mcA\}}}$  $\forall \xx\in\mcX$. 
This gives
\begin{equation*}
Z_0(\xx)
=
\Pprob(G(\xx,\UU) \in \mcA), \quad \xx \in \mcX,
\end{equation*}
thus $Z_0(\xx) \in [0,1]$ is bounded on $\mcX$.

Now fix an index set $A \in \mathcal{P}_{n_U}$, and define the random vector $\UU_A$ with components $U_j$, $j \in A$.
The components of $\UU_A$ contain the inputs for which we perform a sensitivity analysis.
Let $L^2(\Omega, \sigma(\UU_A); L^2(\mcX))$ denote the Bochner space of m.-s. integrable, $L^2(\mcX)$-valued and $\sigma(\UU_A)$-measurable random variables. In other words, random variables in $L^2(\Omega, \sigma(\UU_A); L^2(\mcX))$ depend only on the components $U_j$, $j \in A$ of the random input vector $\UU$.
Then the symbol $\mycev{Z}{\UU_A}$ denotes the conditional expectation of $Z$ with respect to $\UU_A$, to be understood as the orthogonal projection of $Z$ onto $L^2(\Omega, \sigma(\UU_A); L^2(\mcX))$, see \cite[Def.~4.47]{introToComputationalStochasticPDEs}.
This allows us to write
\begin{equation}\label{anova-like}
    \mathds{1}_\Gamma(\xx,\omega)
    =
    Z_0(\xx)+Z_A(\xx,\omega)+Z_{A^c}(\xx,\omega)+Z_{A,A^c}(\xx,\omega), \quad \xx \in \mcX,
\end{equation}
where 
\begin{align}
    Z_A(\xx, \omega) &\coloneqq \mycev{\mathds{1}_{\{G(\xx,\UU(\omega)) \in \mcA\}}}{\UU_A}-Z_0(\xx), \label{Z_A}\\
    Z_{A^c}(\xx, \omega) &\coloneqq \mycev{\mathds{1}_{\{G(\xx,\UU(\omega)) \in \mcA\}}}{\UU_{A^c}}-Z_0(\xx), \\
    Z_{A,A^c}(\xx, \omega) &\coloneqq Z(\xx,\omega)-Z_A(\xx, \omega)-Z_{A^c}(\xx, \omega)-Z_0(\xx).
\end{align}
In particular, we have
\[
Z_A(\xx,\omega)
=
\Pprob(G(\xx,\UU(\omega)) \in \mcA \,\vert\, \UU_A)-\Pprob(G(\xx,\UU) \in \mcA), \quad \xx \in \mcX.
\]
To define the generalized Sobol' index as in \cite{gamboa_SA_for_multiDim_and_funtional_outputs} it remains to show that $Z_A\in L^2(\Omega; L^2(\mcX))$.
The argument for $Z_{A^{c}}$ is analogous.
For $\omega \in \Omega$ we have
\begin{equation*}
\begin{split}
\Vert Z_A(\cdot, \omega) \Vert^2_{L^2(\mcX)}
&=
\int_\mcX (\mycev{Z(\cdot,\omega)}{\UU_A}-Z_0(\xx))^2 \dd \xx\\
&\leq 
\int_\mcX Z_0^2(\xx)\dd \xx 
+
\int_\mcX \mycev{Z(\cdot,\omega)}{\UU_A}^2 \dd \xx<\infty,
\end{split}
\end{equation*}
as both $Z_0$ and $\mycev{Z(\cdot,\omega)}{\UU_A}$, being (conditional) probabilities, are bounded on $\mcX$.
Thus, $\Vert Z_A\Vert^2_{L^2(\Omega; L^2(\mcX))}=\ev{\Vert Z_A(\cdot, \omega)\Vert^2_{L^2(\mcX)}}<\infty$.
{We can thus conclude that $Z=\mathds{1}_{\Gamma}$ and $Z_A=\mycev{1_\Gamma}{\UU_A}-Z_0$ have trace-class covariance operators, denoted by $\mcC_Z$ and $\mcC_{Z_A}$, respectively.}
Then, following Gamboa et al. \cite{gamboa_SA_for_multiDim_and_funtional_outputs} and analogous to Alexanderian et al. \cite{timeDependent_VBSA}, the generalized Sobol' index $S_A$ and the generalized total Sobol' index $S_A^t$ are defined as 
\begin{equation}\label{SA}
    S_A
    :=
    \frac{\tr\mcC_{Z_A}}{\tr \mcC_Z}, \quad A \in \mathcal{P}_{n_U}, \quad \tr \mcC_Z \neq 0,
\end{equation}
and
\begin{equation}\label{StA}
    S^t_A
    :=
    \frac{\tr \mcC_{Z_A}+\tr \mcC_{Z_{A,A^c}}}{\tr \mcC_Z}, \quad A \in \mathcal{P}_{n_U}, \quad \tr \mcC_Z \neq 0.
\end{equation}
Note that the decomposition in \eqref{anova-like} implies $\tr \mcC_Z = \tr \mcC_{Z_A}+\tr \mcC_{Z_{A^c}}+\tr \mcC_{Z_{A,A^c}}$, thus $\tr \mcC_{Z_A}+\tr \mcC_{Z_{A,A^c}}=\tr \mcC_Z - \tr \mcC_{Z_{A^c}}$ and 
\begin{equation}\label{StA:alternative}
S^t_A = 1-\frac{\tr\mcC_{Z_{A^c}}}{\tr \mcC_Z}.
\end{equation}
While the expressions in \eqref{SA} and \eqref{StA} are well-defined, the trace formula in \eqref{trace is integrated variance} gives a practically computable expression for $S_A$ and $S_A^t$.
We call these expressions spatially-integrated Sobol' indices.

\begin{proposition}[Spatially-integrated Sobol' index] \label{proposition:spatially_integrated_sobol_index}
   Let $A \in \mathcal{P}_{n_U}$. 
   Consider the indicator random field $Z=\mathds{1}_\Gamma$ defined in Section~\ref{section:problem_description} and the decomposition in \eqref{anova-like}.
   Let $Z, Z_A, Z_{A^c} \in L^2(\Omega; L^2(\mcX))$.
   Then the spatially-integrated (SpIn) Sobol' index associated with the index set $A$ is given by
    \begin{equation}\label{SA:spin}
        S^{\text{spin}}_A
        :=
        \frac{\int_\mcX \var Z_A(\xx) \dd \xx}{\int_\mcX \var Z(\xx) \dd \xx}
        = S_A,
    \end{equation}
    and the total SpIn Sobol' index associated with $A$ is 
    \begin{equation}\label{StA:spin}
        S^{\text{spin},t}_A
        :=
        {1-}\frac{\int_\mcX \var Z_{A^c}(\xx) \dd \xx}{\int_\mcX \var Z(\xx) \dd \xx}
        =S_A^t.
    \end{equation}
\end{proposition}
\begin{proof}
The identity in \eqref{SA:spin} follows from the trace formula in \eqref{trace is integrated variance}, while \eqref{StA:spin} follows from \eqref{StA:alternative} and again \eqref{trace is integrated variance}.
\end{proof}
It is important to note that the SpIn indices are not novel sensitivity indices: they give a computable expression for the indices proposed in \cite{gamboa_SA_for_multiDim_and_funtional_outputs}. Moreover, as can be seen from the expressions in \eqref{SA:spin} and \eqref{StA:spin}, the SpIn indices are not equal to volume (spatial) integrals of Sobol' indices defined pointwise in $\xx$.  
However, we decided to use the term SpIn to emphasize the role of volume integrals in the construction and interpretation of these indices.

We now propose another an equivalent expression for  $S^{\text{spin}}_A$ and $S^{\text{spin},t}_A$, which allows the use of a pick-freeze scheme \cite{gamboa_pickFreeze,janon2014asymptotic_of_aggr_and_scal_PickFreeze} to estimate them. 
 \begin{proposition}
\label{lemma:spatially_integrated_sobol_pickfreeze}
    Let the setup of Proposition~\ref{proposition:spatially_integrated_sobol_index} hold. 
    Let $A\in\mcP_{n_U}$.
    Let $\UU' \sim \UU$ i.i.d..
    Define the random vector $\UU^{\sim A}\coloneqq(\UU_A,\UU'_{A^c})^\top$, and random field $Z^{\sim A}(\xx,\omega)=\mathds{1}_{\{G(\xx,\UU^{\sim A})\in \mcA\}}$. Then it holds 
    \begin{equation} \label{equation:spin_sobol_index_pickfreeze}
        S^{\text{spin}}_A
        =
        \frac{\int_\mcX\cov{Z(\xx)}{Z^{\sim A}(\xx)}\dd \xx}{\int_\mcX\cov{Z(\xx)}{Z(\xx)}\dd \xx}
    \end{equation}
    and 
    \begin{equation} \label{equation:total_spin_sobol_index_pickfreeze}
        S^{\text{spin},t}_{A}
        =
        {1-}\frac{\int_\mcX\cov{Z(\xx)}{Z^{\sim{A^c}}(\xx)}\dd \xx}{\int_\mcX\cov{Z(\xx)}{Z(\xx)}\dd \xx}.
    \end{equation}
 \end{proposition}
\begin{proof}
    This follows from 
    $\var Z_B(\xx)=\var \mycev{Z(\xx)}{\UU_B}=\cov{Z(\xx)}{Z^{\sim B}(\xx)}$ for $B \in \mcP_{n_U}$ which holds pointwise in $\xx$, see e.g. \cite[Lemma 3.12]{SA_book}. 
\end{proof}
The FEM-based numerical calculation of \eqref{equation:spin_sobol_index_pickfreeze} is presented in Section~\ref{section:SpIn_estimation}.

\section{Estimation of the sensitivity indices} \label{section:estimation_of_sensitivity_indices}

In this section we present the estimators for the sensitivity indices in Sections~\ref{section:kernel_based} and \ref{section:SpIn}. 
Since we focus on PDE-based models discretized by the FEM, we first introduce the notation associated with the spatial discretization. 
Let $\mathcal{T}_h$ be a conforming mesh on $\mcX$ with mesh size $h>0$ and {number of mesh elements $N_h \in \mathbb{N}$}.
Let $\{\varphi_i\}_{i=1}^{N_h}$ denote the corresponding global FE basis functions defined on the mesh $\mathcal{T}_h$.
The FE mass matrix $M \in\mathbb{R}^{N_h\times N_h}$ is given by
$M = \left[\int_\mcX\varphi_i(\xx)\varphi_j(\xx) \dd \xx \right]_{i,j}$, $i,j\in\{1,\dots,N_h\}$. 
In Section \ref{section:HSIC_estimation} and \ref{section:SpIn_estimation} we discuss the estimation of the HSIC-ANOVA and SpIn sensitivity indices, respectively. The corresponding algorithms are listed in the Supplementary Materials, see Algorithm~\ref{SM_algorithm:1_SpInSobol_main} for the HSIC-ANOVA indices  and Algorithm~\ref{SM_algorithm:3_HSIC_main} for the SpIn Sobol' indices.
\subsection{Estimation of the HSIC-ANOVA sensitivity indices}\label{section:HSIC_estimation}

First, we identify suitable kernels associated with the input and output space of the map $\eta$ in \eqref{eta}.

{We start by assuming that the input random variables are $U_i \sim \text{Unif}[0,1]$, $i=1,\dots,n_U$ and, hence, work with the sets $\mcU_i=[0,1]$, $i=1,\dots,n_U$ and $\mcU=[0,1]^{n_U}$.}
{We then consider the so-called 
Order-1 Sobolev Kernel, see  \cite[Eq.~(12)]{daVeiga2021kernel}, defined on $\mcU_i \times \mcU_i$ with the one-dimensional domains $\mcU_i$.  }
\begin{definition} (Order-1 Sobolev Kernel)
    For any $u,u' \in [0,1]$, the Order-1 Sobolev Kernel is given by
    \begin{equation} \label{equation:order_1_sobolev_kernel}
        k_{\text{Sob}}(u,u'):=1+\left(u-\frac{1}{2}\right) \left(u'-\frac{1}{2}\right)+\frac{1}{2}\left[(u-u')^2-|u-u'|+\frac{1}{6}\right].
    \end{equation}
\end{definition}
The {corresponding kernel on $\mcU \times \mcU$ is the} tensor product Order-1 Sobolev kernel 
defined according to \eqref{eq:product_kernel},
\begin{equation*}
k_\text{in}(\uu,\uu')
:=
\prod_{i=1}^{n_U} k_\text{Sob}(u_i,u_i'), \quad \uu, \uu' \in \mcU.
\end{equation*}
As required by Proposition~\ref{proposition:hsic_decomposition}, the kernel $k_\text{in}$ is an ANOVA kernel and satisfies Assumption~\ref{hp:kernel}.
{Moreover, the kernel $k_\text{Sob}$ in \eqref{equation:order_1_sobolev_kernel} is characteristic, see \cite[Prop.~3.6]{sarazin2023new}.}
We work with the kernel $k_\text{in}$ due to its low computational complexity and absence of hyperparameters, which distinguishes it from other ANOVA kernels explored in, e.g., \cite{ginsbourger2016anova}. 
Indeed, the exponential, Gaussian, and Matérn kernels presented in \cite{ginsbourger2016anova} require choosing hyperparameters.
Moreover, their evaluation involves computationally more expensive steps, such as exponentiation or a cumulative density function (CDF) calculation for each input pair of the kernel. 
In contrast, evaluating the Order-1 Sobolev kernel amounts to the evaluation of a quadratic polynomial.

Note  that if the input random variables are not uniformly distributed on $[0,1]$, a transformation {based on the CDF} can be employed. {In particular, in} the numerical experiments in Section~\ref{section:cdr_results} we work with uniform random variables on a generic interval $[a,b]$, $a<b$, and with log-uniform random variables. 
For $Z \sim \text{Unif}[a,b]$, the random variable $\frac{Z-a}{b-a} \sim \text{Unif}[0,1]$. 
Similarly, for $\widetilde{Z} \sim \text{LogUnif}[c,d]$, $0 <c<d$, the random variable $\frac{\log(\widetilde{Z})-\log(c)}{\log(d)-\log(c)} \sim \text{Unif}[0,1]$.

Finally, for the output space $\mathcal{L}^*(\mcX)$ of the map $\eta$ in \eqref{eta}  
we introduce the following set kernel:
    \begin{equation} \label{equation:set_kernel}
        k_{\text{set}}:\mathcal{L}^*(\mcX)\times\mathcal{L}^*(\mcX)\rightarrow\mathbb{R}, \quad k_{\text{set}}(\Gamma_1,\Gamma_2):=\exp\left(-\frac{\lambda(\Gamma_1\Delta \Gamma_2)}{2\sigma^2}\right),
    \end{equation}
    with {$\lambda(\Gamma_1\Delta \Gamma_2)$} as in \eqref{eq:lebesgue_sym_diff}.   
The function $k_\text{set}$ is symmetric and positive definite for any $\sigma\in\mathbb{R}_{>0}$ (i.e.~it is a kernel), continuous, bounded, and characteristic, see \cite{fellmann_kernelSA}, thus satisfying Assumption~\ref{hp:kernel}.
To the best of our knowledge, no other kernels for sets have been proposed in the literature so far.

\subsubsection{The estimator}\label{subsubsectopm:hsic_estimator}
For $A \in \mathcal{P}_{n_U}$, let $k_A(\uu_A,\uu_A'):=\prod_{i \in A} k_\text{Sob}(u_i,u_i')$.
{The $\text{HSIC}(\UU_A,\Gamma)$ can be expressed in terms of expectations of the kernels $k_A$ and $k_\text{set}$.}
Let $\UU_A' \sim \UU_A$ i.i.d. and let $\Gamma'\sim \Gamma$ i.i.d..
Then the following holds \cite[Lemma 1]{Gretton2005_CrossCovIntro}
\begin{equation*}
\begin{split}
\text{HSIC}(\UU_A,\Gamma) 
&=
\ev{(k_A(\UU_A,\UU_A')k_\text{set}(\Gamma,\Gamma')}
+
\ev{k_A(\UU_A,\UU_A')}\ev{k_\text{set}(\Gamma,\Gamma')}\\
&-
2\cev{\cev{k_A(\UU_A,\UU_A')}{\UU_A'}\cev{k_\text{set}(\Gamma,\Gamma')}{\Gamma'}}{\UU_A, \Gamma}. 
\end{split}
\end{equation*}
Since $k_A$ is an ANOVA kernel, \eqref{anova:plus one} together with the property in \eqref{anova:centered}, implies $\ev{(k_A(\UU_A,\UU_A')}=1$. Thus, the above formula simplifies to 
\begin{equation*}
\begin{split}
\text{HSIC}(\UU_A,\Gamma)  &=
\ev{(k_A(\UU_A,\UU_A')-1)k_\text{set}(\Gamma,\Gamma')}.
\end{split}
\end{equation*}
This last expression motivates the following estimator.
Let $(\UU_A^{(i)},\Gamma^{(i)} ) \sim (\UU_A,\Gamma)$, $i=1,\dots,n$, be i.i.d. input-output pairs, each following the joint distribution $\Pprob_{\UU_A,\Gamma}$. 
The estimator of $\text{HSIC}(\bm{U}_A,\Gamma)$ proposed in \cite{fellmann_kernelSA} is given by
    \begin{equation}\label{equation:hsic_estimation}
        \begin{aligned}
            \widehat{\text{HSIC}}(\bm{U}_A,\Gamma) &= 
            \frac{2}{{n(n-1)}}\sum^n_{j=2}\sum^{j-1}_{i=1}{K}^{(i,j)}_{A}{K}^{(i,j)}_{\Gamma}, 
        \end{aligned}
    \end{equation} 
where ${K}^{(i,j)}_{A}$ and ${K}^{(i,j)}_{\Gamma}$ denote the $(i,j)$-th entries of the matrices $K_A, K_\Gamma \in\mathbb{R}^{n\times n}$, respectively, defined as 
    \begin{align*}
        K_A &\coloneqq \Bigg[ (1-\delta_{ij})\bigg(\bigg[\prod_{l\in A} k_\text{Sob}(U_{l}^{(i)},U_{l}^{(j)})\bigg]-1\bigg)\Bigg]_{i,j}, \quad i,j\in\{1,\dots,n\}, \\  
        K_\Gamma &\coloneqq \Bigg[(1-\delta_{ij})k_\text{set}(\Gamma^{(i)},\Gamma^{(j)})\Bigg]_{i,j}, 
        \quad i,j\in\{1,\dots,n\}.
    \end{align*}
    Note that the diagonal entries of the matrices $K_A$ and $K_\Gamma$ are set equal to zero since these entries do not appear in the estimator in \eqref{equation:hsic_estimation}.
    Moreover, the estimator in \eqref{equation:hsic_estimation} is unbiased. 
    Plugging in the expression for the set kernel \eqref{equation:set_kernel} we can rewrite $K_\Gamma$ as 
    \begin{equation*}
    K_\Gamma= \Bigg[(1-\delta_{ij})\exp{\bigg( -\frac{\Lambda_{ij}}{2\sigma^2} \bigg)}\Bigg]_{i,j}, \quad i,j\in\{1,\dots,n\}
    \end{equation*}
    with the matrix $\Lambda \in \mathbb{R}^{n \times n}$ given by
    \begin{equation}\label{Lambda}
    \Lambda = \Bigg[ \int_\mcX (\mathds{1}_{\Gamma^{(i)}}-\mathds{1}_{\Gamma^{(j)}})^2 \dd\lambda \Bigg]_{i,j}, \quad i,j\in\{1,\dots,n\}.
    \end{equation}
    Note that the set kernel involves the scaling parameter $\sigma>0$ which we are free to choose.
    Following \cite[Section 4.1]{fellmann_kernelSA} we define $\sigma$ as 
    \begin{equation*}
    \sigma^2=\frac{1}{n^2}\sum_{i,j=1}^n\Lambda_{ij}.
    \end{equation*}
Finally, the estimator for the HSIC-ANOVA index is defined as
\begin{equation}\label{HSIC final estimate}
\widehat S_A^{HSIC} 
=
\frac{\widehat{\text{HSIC}}(\UU_A,\Gamma)}{\widehat{\text{HSIC}}(\UU,\Gamma)}.
\end{equation}

\subsubsection{{Estimator cost}}
To compute all $n_U$ first-order HSIC-ANOVA indices $S_A^{HSIC}$ (where $|A|=1$) with $n$ Monte Carlo samples, we require $n$ evaluations of the process model $G$ to generate i.i.d. input-output pairs $(\UU^{(i)}, \Gamma^{(i)}) \sim (\UU, \Gamma)$, $i=1, \dots,n$.
With this we compute the denominator in \eqref{HSIC final estimate}.
Realizations of $\UU_A^{(i)}$ can be obtained by selecting the components $U_\ell^{(i)}$ with index $\ell \in A$ in the random vector $\UU^{(i)}$.
With this we evaluate the numerator in \eqref{HSIC final estimate}. 
The estimator cost is independent of the number of inputs $n_U$.
Note that the expression in \eqref{equation:hsic_estimation} is a double summation over $n$ indices, however, we only require $n$ model evaluations in this expression.

\subsubsection{FE approximation of the volume integral in the set kernel}\label{sec:FE:volume:set kernel}
The construction of $K_A$ requires evaluations of the kernel $k_\text{Sob}$ which is defined in terms of elementary functions. 
Instead, building $K_\Gamma$ requires evaluations of the set-kernel in \eqref{equation:set_kernel}, which involves a volume integral over $\mcX$, see \eqref{Lambda}. 
In \cite{fellmann_kernelSA}, the authors propose a simple Monte Carlo approximation for the volume integral. 
This Monte Carlo estimator is nested inside the outer Monte Carlo estimator of the sensitivity index.
This approach is inefficient: For each pair of random sets $\Gamma^{(i)}$ and $\Gamma^{(j)}$ in the HSIC-ANOVA estimator in \eqref{equation:hsic_estimation}, one needs to perform a Monte Carlo estimation in $\mcX$ to approximate the Lebesgue measure of the pair's symmetric difference. 

Since in this work we focus on PDE-based models discretized by the FEM, we can exploit the spatial discretization of $\mcX$ provided by the FE mesh and obtain a direct numerical approximation of the volume integral without additional Monte Carlo samples. 
We proceed as follows. 
For each realization $\Gamma^{(i)}$ of the random set $\Gamma$ we approximate the indicator function $\mathds{1}_{\Gamma^{(i)}}$ in the FE basis.
The approximation is $\mathds{1}_{\Gamma^{(i)},h}=\sum_{k=1}^{N_h} c_k^{(i)}\varphi_k$, and the coefficients are determined by the interpolation conditions $\mathds{1}_{\Gamma^{(i)},h}(\xx_m)=\mathds{1}_{\Gamma^{(i)}}(\xx_m)$, $m=1,\dots,n_\text{vert}$, where $\xx_m$ is a vertex of the FE mesh.
For example, in Section~\ref{section:cdr_results} we employ FE with piecewise linear, continuous shape functions on $\mcX$.
In this case $n_\text{vert}=N_h$ and $c^{(i)}_k=\mathds{1}_{\Gamma^{(i)}}(\xx_k)$, $k=1,\dots,N_h$, are uniquely determined.
Finally, replacing $\mathds{1}_{\Gamma^{(i)}}$ by its FE approximation $\mathds{1}_{\Gamma^{(i)},h}$ in \eqref{Lambda} allows us to evaluate the volume integral exactly as the inner product of the coefficient vectors weighted by the FE mass matrix:
 \begin{align*}
        \int_{\mcX}(\mathds{1}_{\Gamma^{(i)},h}(\xx)-\mathds{1}_{\Gamma^{(j)},h}(\xx))^2 \dd \xx 
        &= \int_{\mcX} \left(\sum_{k=1}^{N_h} c^{(i)}_k\varphi_k(\xx)-\sum_{k=1}^{N_h}c_k^{(j)}\varphi_k(\xx) \right)^2 \dd \xx \\
        &= \int_{\mcX}\bigg( \sum_{k=1}^{N_h}(c^{(i)}_k-c^{(j)}_k)\varphi_k(\xx)\bigg)\bigg( \sum_{\ell=1}^{N_h}(c^{(i)}_\ell-c^{(j)}_\ell)\varphi_\ell(\xx)\bigg) \dd \xx \\
        &= \sum_{k=1}^{N_h}\sum_{\ell=1}^{N_h}(c^{(i)}_k-c^{(j)}_k)(c^{(i)}_\ell-c^{(j)}_\ell)\int_{\mcX}\varphi_k(\xx)\varphi_\ell(\xx) \dd \xx \\
        &= (\cc^{(i)}-\cc^{(j)})^{\top}M(\cc^{(i)}-\cc^{(j)}),
    \end{align*}   
{where $\cc^{(i)} = [c_1^{(i)},\ldots,c_{N_h}^{(i)}] \in\mathbb{R}^{N_h}$ is the vector of coefficients of the FE approximation of $\mathds{1}_{\Gamma^{(i)},h}$, and analogously for $\cc^{(j)}$.}
\subsection{{Estimation of the SpIn sensitivity indices}}\label{section:SpIn_estimation}
We discuss the estimation of the generalized Sobol' index $S_A^\text{spin}$ given in \eqref{equation:spin_sobol_index_pickfreeze} for PDE-based process models discretized by the FEM. 
The estimator for the generalized total Sobol' index $S_A^{\text{spin},t}$ is analogous.
Each estimator has two parts: the estimation of the covariance term and the approximation of the volume integral over $\mcX$.

\subsubsection{The pick-freeze estimator}
We estimate the covariance terms in \eqref{equation:spin_sobol_index_pickfreeze} with a pick-freeze scheme \cite{gamboa_pickFreeze, janon2014asymptotic_of_aggr_and_scal_PickFreeze} with $n\in \mathbb{N}$ Monte Carlo samples each.
To this end, let
    $(\UU^{(i)}_{I})_{i=1}^n$, $(\UU^{(i)}_{II})_{i=1}^n$ and $(\UU^{(i)}_{III})_{i=1}^n$ denote $3n$ i.i.d. {copies}     of the random vector $\UU$, organized in the arrays $\UU_{I}$, $\UU_{II}$ and $\UU_{III}$, respectively.
    Let $\UU_{\sim A}^{(i)}:=(\UU_{II,A}^{(i)}, \UU_{III,A^c}^{(i)})$ and define
    \begin{align}
    Z_I^{(i)}(\xx)&\coloneqq\mathds{1}_{\{G(\xx,\UU_I^{(i)})\in \mcA\}},\quad
    Z_{II}^{(i)}(\xx)\coloneqq\mathds{1}_{\{G(\xx,\UU_{II}^{(i)})\in \mcA\}},\label{eq:ZIZII}\\
    Z_{\sim A}^{(i)}(\xx)&\coloneqq\mathds{1}_{\{G(\xx,\UU_{\sim A}^{(i)})\in \mcA\}}, \quad i=1,\dots,n.\label{eq:ZnotA}
    \end{align}
    Recall that $\cov{Z(\xx)}{Z^{\sim A}(\xx)}=\ev{Z(\xx)Z^{\sim A}(\xx)}-\ev{Z(\xx)}\ev{Z^{\sim A}(\xx)}$ for each $\xx \in \mcX$.
This motivates the estimator
\begin{equation}\label{spin:estimator}
\widehat{S}^\text{spin}_A
:=
\frac{\int_\mcX \left(\frac1n\sum_{i=1}^n Z_{II}^{(i)}(\xx) Z_{\sim A}^{(i)}(\xx) - \frac1n \sum_{i=1}^n Z_{II}^{(i)}(\xx) \frac1n \sum_{i=1}^n Z^{(i)}_{\sim A}(\xx)\right)\dd \xx}{\int_\mcX \left(\frac1n\sum_{i=1}^n Z_{I}^{(i)}(\xx) Z_{I}^{(i)}(\xx) - \frac1n \sum_{i=1}^n Z_{I}^{(i)}(\xx) \frac1n \sum_{i=1}^n Z^{(i)}_{I}(\xx)\right)\dd \xx}.
\end{equation}

\subsubsection{{Estimator cost}}
To compute all $n_U$ first-order spatially-integrated Sobol' indices $S_A^\text{spin}$ (where $|A|=1$) with $n$ Monte Carlo samples, we require $n \cdot (2\cdot  n_U+1)$ evaluations of the process model $G$.
For each Monte Carlo sample, one input-output pair is used to calculate a realization of the indicator random field $Z_I^{(i)}$ in \eqref{eq:ZIZII}.
The remaining input-output pairs are used to calculate realizations of $Z_{II}^{(i)}$ in \eqref{eq:ZIZII} and $Z_{\sim A}^{(i)}$ in \eqref{eq:ZnotA} for each of the $n_U$ index sets $A$ with $|A|=1$. 
Note that the estimator cost grows linearly in the number of inputs $n_U$.

\subsubsection{FE approximation of the volume integral}
Crucially, the volume integrals in \eqref{spin:estimator} can be approximated using the FE mesh and FE basis as in Subsection~\ref{sec:FE:volume:set kernel}. 
The random field realizations required in \eqref{spin:estimator} are replaced by approximations in the given FE basis. The approximations read 
\begin{align}
Z_{I,h}^{(i)}(\xx)
&=
\sum_{k=1}^{N_h} c^{(i)}_{I,k} \varphi_k(\xx), \quad
Z_{II,h}^{(i)}(\xx)
=
\sum_{k=1}^{N_h} c^{(i)}_{II,k} \varphi_k(\xx), \quad
Z_{\sim A,h}^{(i)}(\xx)
=
\sum_{k=1}^{N_h} c^{(i)}_{\sim A,k} \varphi_k(\xx).
\end{align}
The coefficients $\cc_{I}^{(i)}:=(c_{I,1}^{(i)}, \dots, c_{I,N_h}^{(i)})^\top$, $\cc_{II}^{(i)}:=(c_{II,1}^{(i)}, \dots, c_{II,N_h}^{(i)})^\top$ and $\cc_{\sim A}^{(i)}:=(c_{\sim A,1}^{(i)}, \dots, c_{\sim A,N_h}^{(i)})^\top$ are determined by interpolation conditions, see Subsection~\ref{sec:FE:volume:set kernel}. 
Finally, inserting these FE approximations into \eqref{spin:estimator} and rearranging gives the FEM-SpIn estimator
\begin{equation}\label{fem:spin:estimator}
\widehat S_{A, h}^\text{spin}
:=
\frac{\frac1n \sum_{i=1}^n {\cc^{(i)}_{II}}^\top M \cc^{(i)}_{\sim A}-\left(\frac1n \sum_{i=1}^n (\cc_{II}^{(i)})^\top\right) M \left(\frac1n\sum_{i=1}^n \cc_{\sim A}^{(i)}\right)}{\frac1n \sum_{i=1}^n {\cc^{(i)}_{I}}^\top M \cc^{(i)}_{I}-\left(\frac1n \sum_{i=1}^n (\cc_{I}^{(i)})^\top\right) M \left(\frac1n\sum_{i=1}^n \cc_{I}^{(i)}\right)},
\end{equation}
where $M$ is the mass matrix associated with the FE basis functions.

Note that the estimator in \eqref{fem:spin:estimator} is biased and can be affected by rounding errors if the two quantities subtracted are of similar size.
However, when using piecewise linear, continuous FE shape functions, we do not expect cancellations by rounding errors since in this case the coefficient vectors $\cc_I$, $\cc_{II}$ and $\cc_{\sim A}$ are realizations of binary random vectors with entries either equal to zero or equal to one.
Thus, the quantities in the subtraction differ by a factor of order $\mathcal{O}(1/n)$.

\section{Numerical experiments}\label{section:cdr_results}
In this section we study a PDE-based process model, namely a hydrogen combustion process.
We present the model equations and parameter values in Section~\ref{subsection:experims_model_systems_cdr}, the numerical discretization in Section~\ref{subsection:numerical_setup}, and the corresponding feasible spatial sets and observation windows in Section~\ref{subsection:feasible set}.
Section~\ref{subsection:numerical_results} discusses several numerical studies, in particular a comparison of the HSIC-ANOVA and the SpIn sensitivity indices and their performance for different spatial observation windows.

\subsection{Process model}\label{subsection:experims_model_systems_cdr}

We consider the hydrogen combustion process $2H_2+O_2\rightarrow2H_2O$ introduced in \cite{paper:CDR_paper2, paper:CDR_paper1}.
This process is modelled by a coupled system of nonlinear convection-diffusion-reaction (CDR) PDEs.
The PDE solution has four components, $\yy = \yy(t,\xx) = (Y_F,Y_O,Y_P,Y_T)^\top$, where  
$Y_F$ is the mass fraction of the fuel (hydrogen), $Y_O$ is the mass fraction of the oxidizer (oxygen), and $Y_P$ is the mass fractions of the product (water).
$Y_T$ denotes the temperature field. 
The PDEs are posed on the domain $\mcX= (0.0,0.5)\times(0.0,1.0) \subset \mathbb{R}^2$ (see Figure~\ref{fig:CDR_spat_dom}) and read
\begin{equation} \label{eqn:experims_models_cdr_pde}
    \frac{\partial y_i}{\partial{t}}
    = 
    \kappa\,\Delta y_i- \ww \cdot \nabla y_i+\vs_i(\yy), \quad i\in \{F,O,P,T\}.
\end{equation}
In \eqref{eqn:experims_models_cdr_pde} $\kappa \in \mathbb{R}$ is the molecular diffusivity and $\ww \in \Rset^2$ is a divergence-free velocity field. 
The reaction term $\vs(\yy) = (s_F, s_O, s_P, s_T)^\top \in \mathbb{R}^4$ uses an Arrhenius-type expression,
\begin{equation} \label{eqn:experims_models_cdr_pde_nonlinearReactionSourceTerm}
    \begin{aligned} 
        s_i(\yy) 
        &= 
        -\nu_i\left(\frac{W_i}{\rho}\right)\left(\frac{\rho{Y_F}}{W_F}\right)^{\nu_F}\left(\frac{\rho{Y_O}}{W_O}\right)^{\nu_O}A\,{\exp}\left(-\frac{E}{RT}\right),\;\; i\in\{F,O\}, \\
        s_P(\yy) 
        &= 
        \nu_P\left(\frac{W_P}{\rho}\right)\left(\frac{\rho{Y_F}}{W_F}\right)^{\nu_F}\left(\frac{\rho{Y_O}}{W_O}\right)^{\nu_O}A\,{\exp}\left(-\frac{E}{RT}\right), \\
        s_T(\yy) 
        &= 
        Q \cdot s_P(\yy),
    \end{aligned}
\end{equation}
where $A$ and $E$ are the pre-exponential factor and the activation energy, respectively. 
Further parameters are the stoichiometric coefficients $\nu_i$ of the combustion process, the molecular weights $W_i$, the density of the mixture $\rho$, the heat of reaction $Q$, and the universal gas constant $R$. 
\begin{figure}
    \centering
    \includegraphics[width=0.6\linewidth]{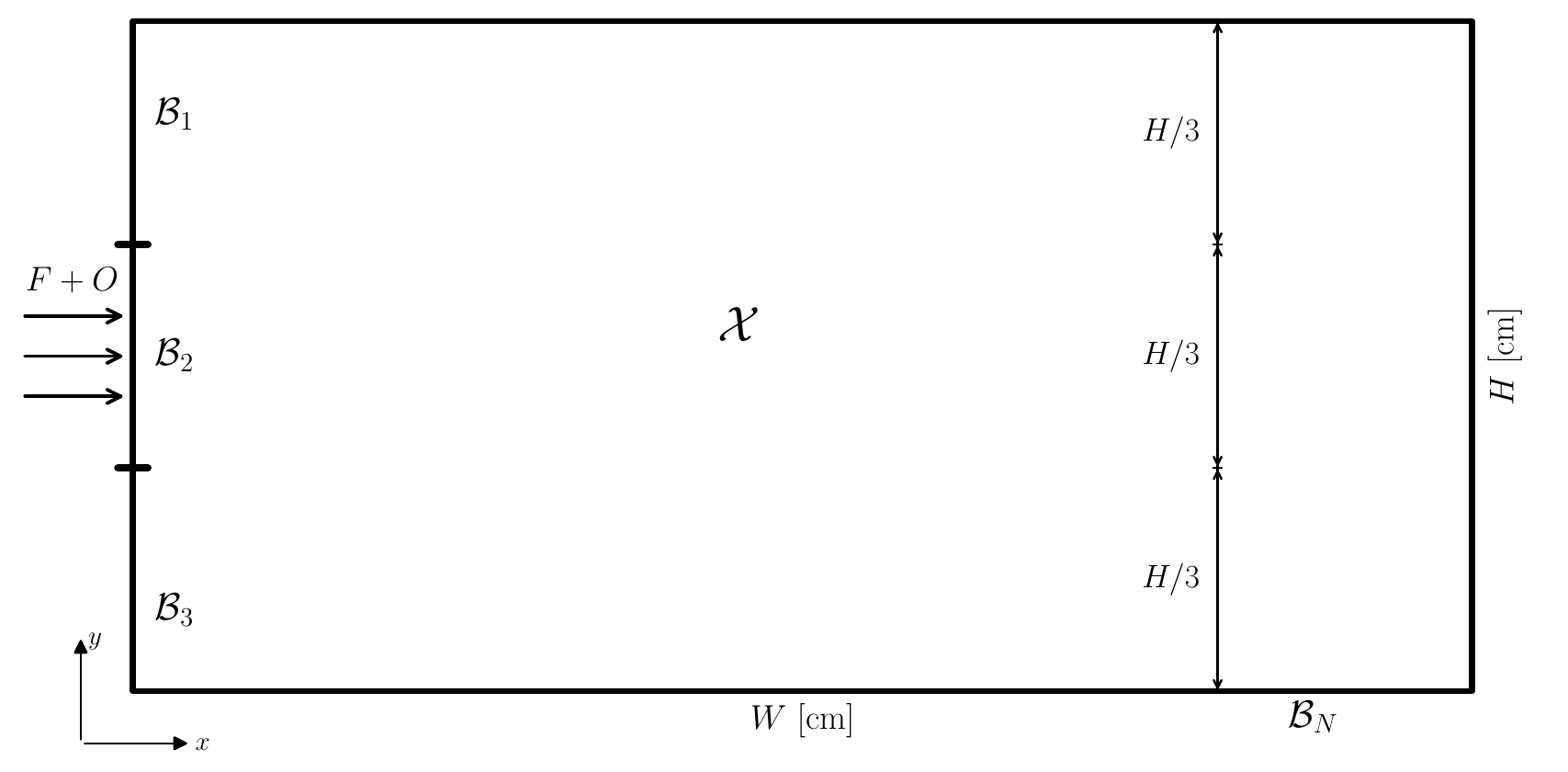}
    \caption{The spatial domain $\mcX$ of the CDR model with boundary $\mcB =\mcB_1\cup\mcB_2\cup\mcB_3 \cup \mcB_N$.}
    \label{fig:CDR_spat_dom} 
\end{figure}
The initial and boundary conditions are
\begin{equation}\label{eqn:experims_models_cdr_pde_BCs_ICs}
    \begin{aligned}
        \yy(0,\xx) &= \yy_0(\xx), \quad &&\forall \xx \in \mcX,&&\\
        \yy(t,\xx)&=\yy_j(\xx), \quad &&\forall t>0, \quad &&\forall \xx \in \mcB_j, \quad j=1,2,3,\\ \nn \cdot \nabla y_i(t,\xx) &= 0, \quad &&\forall t>0, \quad &&\forall \xx \in \mcB_N, \quad i=1,2,3,4, 
    \end{aligned}
\end{equation}
where the boundary $\partial \mcX =: \mcB = \mcB_1 \cup \mcB_2 \cup \mcB_3 \cup \mcB_N$ is split into three Dirichlet and one Neumann part, respectively  (see Figure~\ref{fig:CDR_spat_dom}).
At the initial time $t=0$ the domain is empty with a temperature of 300 Kelvin, that is, $\yy_0=(0,0,0,300)^\top$ in \eqref{eqn:experims_models_cdr_pde_BCs_ICs}.
The conditions on the different portions of the Dirichlet boundary are 
\begin{equation}\label{IC more}
    \begin{aligned}
        \yy_1 & = (0,0,0, T_o)^\top, \\
        \yy_2 & = (Y_{F}^{(0)}(\phi), 0.2259, 0, T_i)^\top, \\
        \yy_3 & = (0,0,0, T_o)^\top,
    \end{aligned}
\end{equation}
where $T_i$ is the constant inflow temperature, $T_o$ is the constant boundary temperature next to the inflow boundary, and the oxidizer fraction at the inflow is set to $0.2259$.
In \eqref{IC more} we define $Y_{F}^{(0)}(\phi)$ via the fuel-to-oxidizer ratio $\phi$ as
\begin{align*}
    Y_{F}^{(0)}(\phi) := \frac{\phi\cdot 0.2259}{8},
\end{align*}
where we used the oxidizer fraction $0.2259$ at the inflow, and $8$ is the mass stoichiometric ratio.
Finally, we impose homogeneous Neumann boundary condition on $\mcB_N$.
\begin{table}
    \centering
    \begin{tabular}{c|c|c}
        Parameter & Value/Distribution & Unit \\
        \hline
        $\kappa$ & 2.0 & $\frac{\text{cm}^2}{\text{sec}}$ \\
        $\ww$ & $(50,0)^\top$ & $\frac{\text{cm}}{\text{sec}}$ \\
        $A$ & LogUnif [5.5e11, 1.5e12] & $\frac{\text{cm}^6}{{\text{mol}^2}\cdot\text{sec}}$\\
        $E$ & LogUnif [1.5e3, 9.5e3] & $\frac{\text{J}}{\text{mol}}$\\
        $\nu_i$ & $\nu_F=2, \nu_O=1, \nu_P=2$ & [--] \\
        $W_i$ & $W_F = 2.016, W_O = 31.9, W_P = 18 $ & $\frac{\text{gram}}{\text{mol}}$ \\
        $\rho$ & 1.39\text{e-}3 & $\frac{\text{gram}}{\text{cm}^3}$ \\
        $Q$ & 9800 & K \\
        $R$ & 8.314472 & $\frac{\text{J}}{\text{mol}\cdot{\text{K}}}$\\
        $T_o$ & Unif [200, 400] & K \\
        $T_i$ & Unif [850, 1000] & K \\
        $\phi$ & Unif [0.5, 1.5] & [--]
    \end{tabular}
    \caption{Parameters in CDR model.}
    \label{tab:CDR_parameters}
\end{table}
Table~\ref{tab:CDR_parameters} lists all the parameters in the CDR model. 
We model five of them, $A, E, T_i, T_o, \phi$ as random variables as in \cite{paper:CDR_paper1}.
We collect them in the random vector $\UU=(A,E,T_i,T_o,\phi)^\top $ such that $\yy = \yy(t,\xx,\UU)$.

\subsection{Discretization and numerical solution} \label{subsection:numerical_setup}

After sampling the random input vector $\UU$ according to the relevant distributions in Table~\ref{tab:CDR_parameters}, we obtain a CDR problem with deterministic coefficients which we solve with the FEM in space and the Crank--Nicolson method \cite{Crank_Nicolson_1947} for the time integration. 
The simulation stops after $t=0.05$ seconds in physical time.
We use $500$ time steps, amounting to a resolution of $\Delta t = 10^{-4}$ seconds per step.
For the spatial domain $\mcX$ in Figure~\ref{fig:CDR_spat_dom}, the width {$W=1$} [cm] and the height {$H=0.5$} [cm]. 
We generated a triangular mesh of $\mcX$ using the software GMSH \cite{paper:gmsh} with a mesh resolution of $h=0.025$ [cm].
The final mesh contains $1,868$ triangular cells and $995$ vertices. 
Each scalar field $Y_F$, $Y_O$, $Y_P$, and $T$ is discretized 
using piecewise linear, continuous (CG1) finite element shape functions in the software {FEniCS (Version 2019.1.0)} \cite{fenics2015}. 
This results in $995$ degrees of freedom (dofs) for each scalar field and $4 \times 995=3,980$ dofs in total. 
In each time step we solve a nonlinear system of equations using Newton's method.
Therein, we apply the GMRES iterative solver \cite{gmres} with an Incomplete LU preconditioner \cite{saad2003_iterativeMethodsForSparseLinearSystems_ILU} to solve the linearized system (with nonsymmetric coefficient matrix) in each iteration.

\subsection{Feasible set and observation windows}\label{subsection:feasible set}
For the SA we focus on the temperature field $Y_T=Y_T(t,\xx,\UU)$. 
We are interested in the setting where the temperature remains below a critical threshold value $T_\text{crit}>0$ at a specified time $t_\text{obs}>0$ and for specified points $\xx \in \mcX_\text{obs}\subseteq \mcX$ in an observation window within the combustion domain.
This defines the feasible set in \eqref{equation:feasible_set} as
\begin{equation}
\label{equation:excursion_set_for_cdr_temp_output}
\Gamma(\uu)
=
\{\xx \in \mcX_\text{obs}\colon Y_T(t=t_\text{obs},\xx,\uu)\leq T_\text{crit}\}, \quad \uu \in \mcU.
\end{equation}
For demonstration purposes we choose $t_\text{obs}=0.05$ [s] and $T_\text{crit}=700$ [K] by observing the average temperature field in a few simulation runs.
{One can replace the inequality constraint in \eqref{equation:excursion_set_for_cdr_temp_output} by other physically meaningful critical values.}
The indicator function associated with the feasible set in \eqref{equation:excursion_set_for_cdr_temp_output} reads
\begin{equation*}
    \mathds{1}_{\Gamma}\colon \mcX \times \mcU \rightarrow \Rset, \qquad (\xx,\uu) \mapsto \mathds{1}_{\{\xx \in \mcX_\text{obs}\colon Y_T(t=t_\text{obs}, \xx, \uu)\leq T_\text{crit}\}}.
\end{equation*}

\subsection{Numerical results} \label{subsection:numerical_results}

First we plot some realizations of the temperature field and show the indicator function of the corresponding feasible sets. 
Next we study the HSIC-ANOVA and SpIn Sobol' indices for different observation windows $\mcX_\text{obs}$ in the spatial domain $\mcX$.

\subsubsection{Temperature field and indicator function} \label{subsubsection:temp_field_and_indicator_functions}
In the left panel of Figures~\ref{fig:temp_and_indicator_47}, \ref{fig:temp_and_indicator_48}, and \ref{fig:temp_and_indicator_ab}, we plot the temperature field for different realizations $\uu$ of the random input vector $\UU$. 
The right panel of each of these figures plots the indicator function of the feasible set \eqref{equation:excursion_set_for_cdr_temp_output} approximated in the FEM basis. Note that since we use CG1 elements in our FEM setup, the indicator function is equal to one at each vertex $\xx_m$ of the FE mesh which satisfies $\xx_m \in \Gamma(\uu)$ (colored vertices in the figures) and is equal to zero at all other vertices.
We observe that the indicator function, respectively the feasible set, can vary substantially depending on the input parameters. 
For example, in Figure~\ref{fig:temp_and_indicator_48} the feasible set occupies most of the spatial domain because the input parameters result in a relatively cold flame. 
On the other hand, Figure~\ref{fig:temp_and_indicator_ab} corresponds to a relatively hot flame, and thus the feasible set occupies a much smaller region in the spatial domain. 

\newcommand{\twocolumnlength}{0.45\textwidth}
\begin{figure}[!htb]
    \centering
    \begin{subfigure}[b]{\twocolumnlength}%
        \centering
        \includegraphics[width=\linewidth]{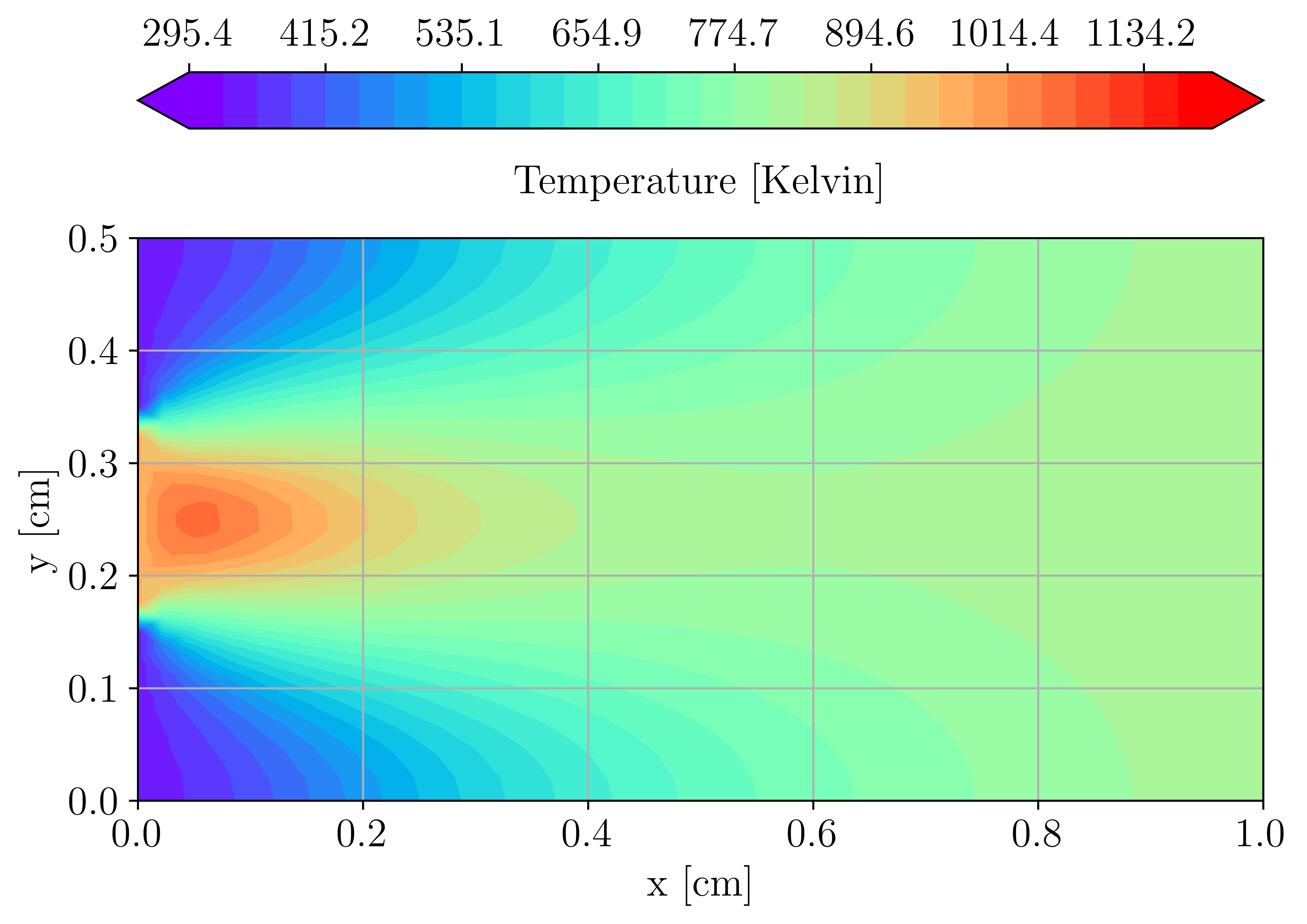}
    \end{subfigure}%
    \hfill
    \begin{subfigure}[b]{\twocolumnlength}%
        \centering
        \includegraphics[width=\linewidth]{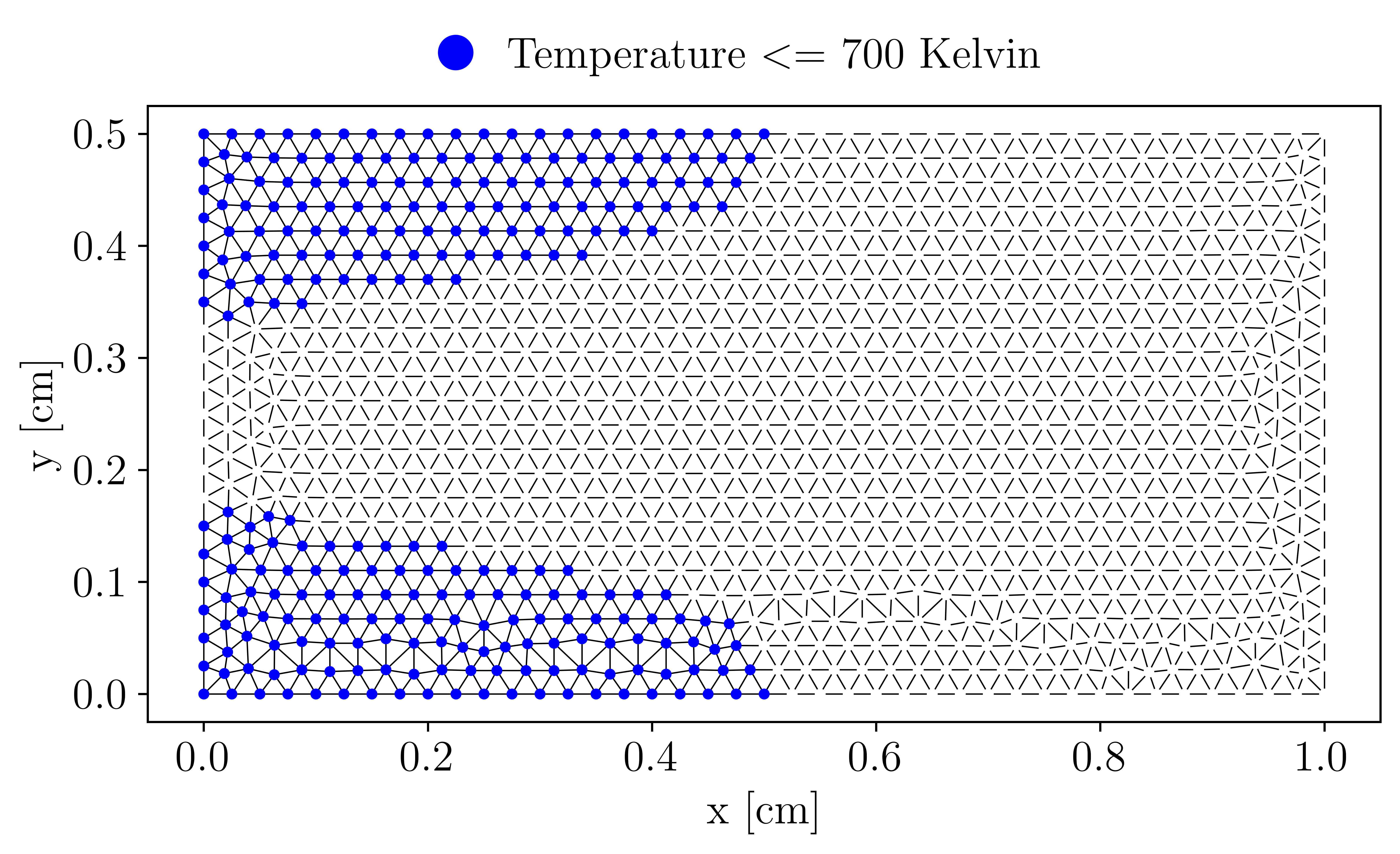}
    \end{subfigure}%
    \caption{Temperature field (left) and associated indicator function (right) of the feasible set for $t_\text{obs}=0.05$ [s], $T_\text{crit}=700$ [K] and the input parameters $(A,E,T_i,T_o,\phi)=(5.8134 \times 10^{11}, 4.4688 \times 10^{3},9.6086 \times 10^{2}, 3.3876\times 10^2, 1.2222\times 10^0)$. 
    }
    \label{fig:temp_and_indicator_47}
\end{figure}
\begin{figure}[!htb]
    \centering

    \begin{subfigure}[b]{\twocolumnlength}%
        \centering
        \includegraphics[width=\linewidth]{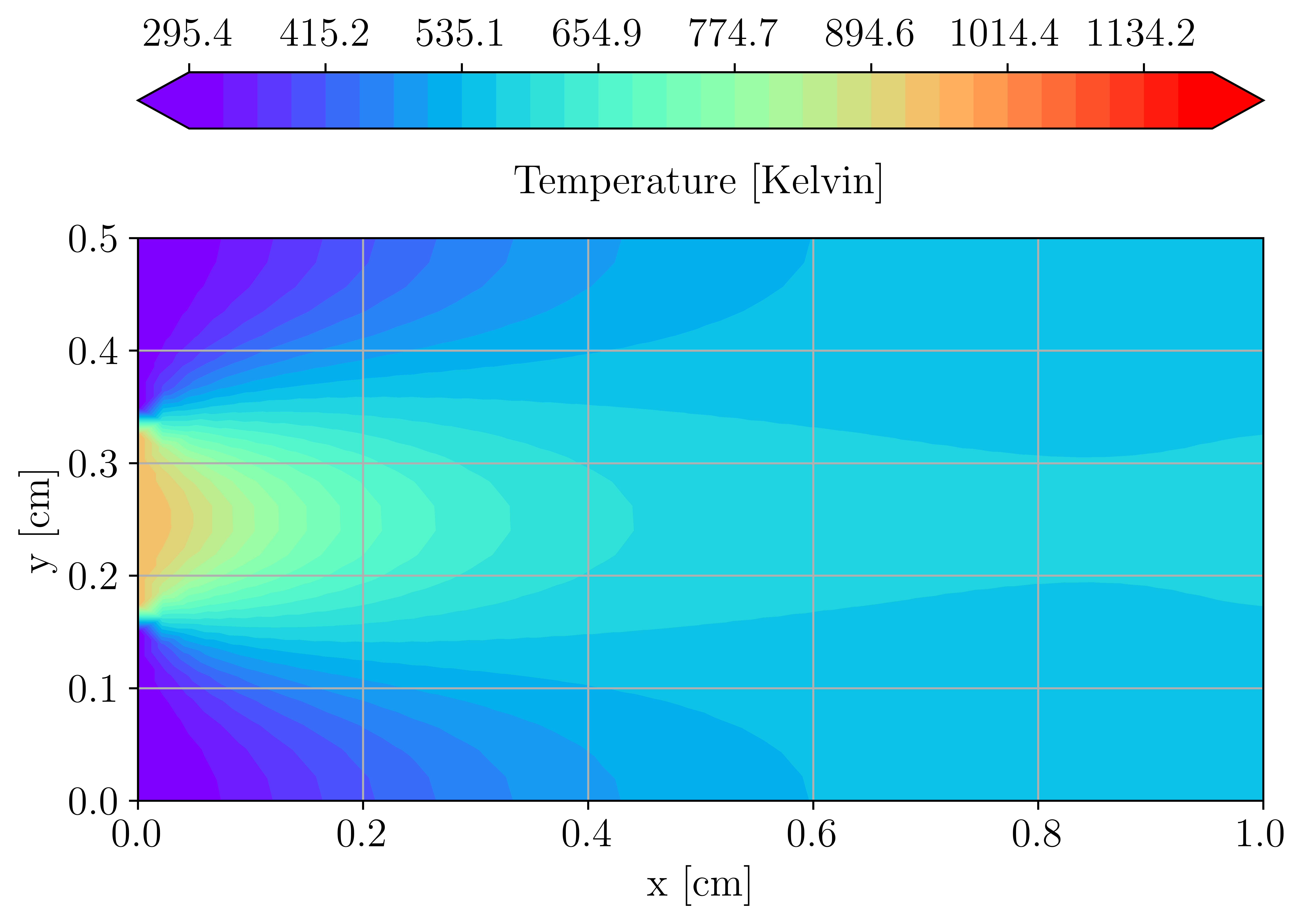}
    \end{subfigure}%
    \hfill
    \begin{subfigure}[b]{\twocolumnlength}%
        \centering
        \includegraphics[width=\linewidth]{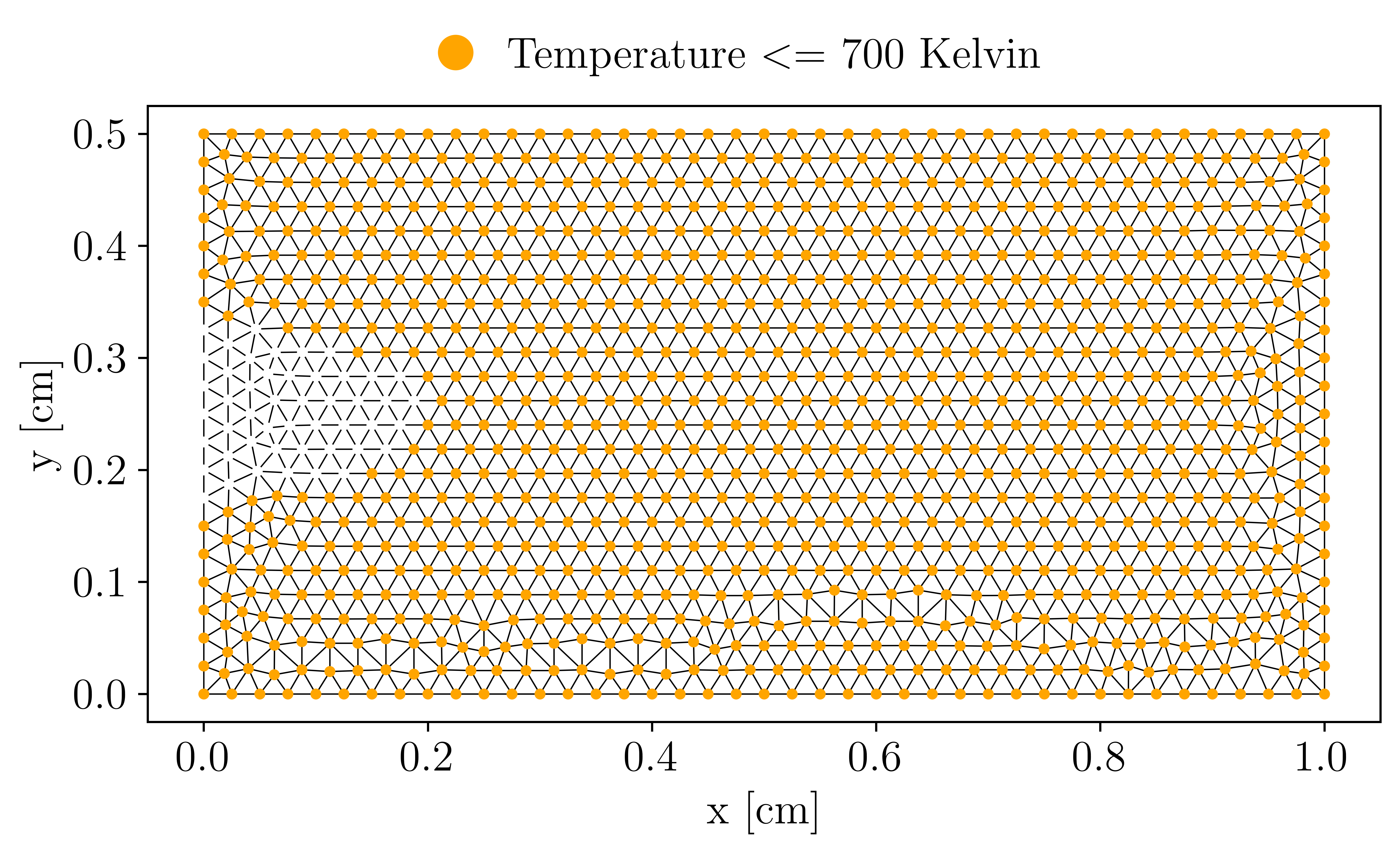}
    \end{subfigure}%
    
    \caption{Temperature field (left) and associated indicator function (right) of the feasible set for $t_\text{obs}=0.05$ [s] seconds, $T_\text{crit}=700$ [K] and the input parameters $(A,E,T_i,T_o,\phi)=(6.1999 \times 10^{11}, 8.3363 \times 10^{3}, 9.5296 \times 10^{2}, 2.9541 \times 10^{2}, 5.7396 \times 10^{-1})$. 
    }
    \label{fig:temp_and_indicator_48}
\end{figure}
\begin{figure}[!htb]
    \centering
    \begin{subfigure}[b]{\twocolumnlength}%
        \centering
        \includegraphics[width=\linewidth]{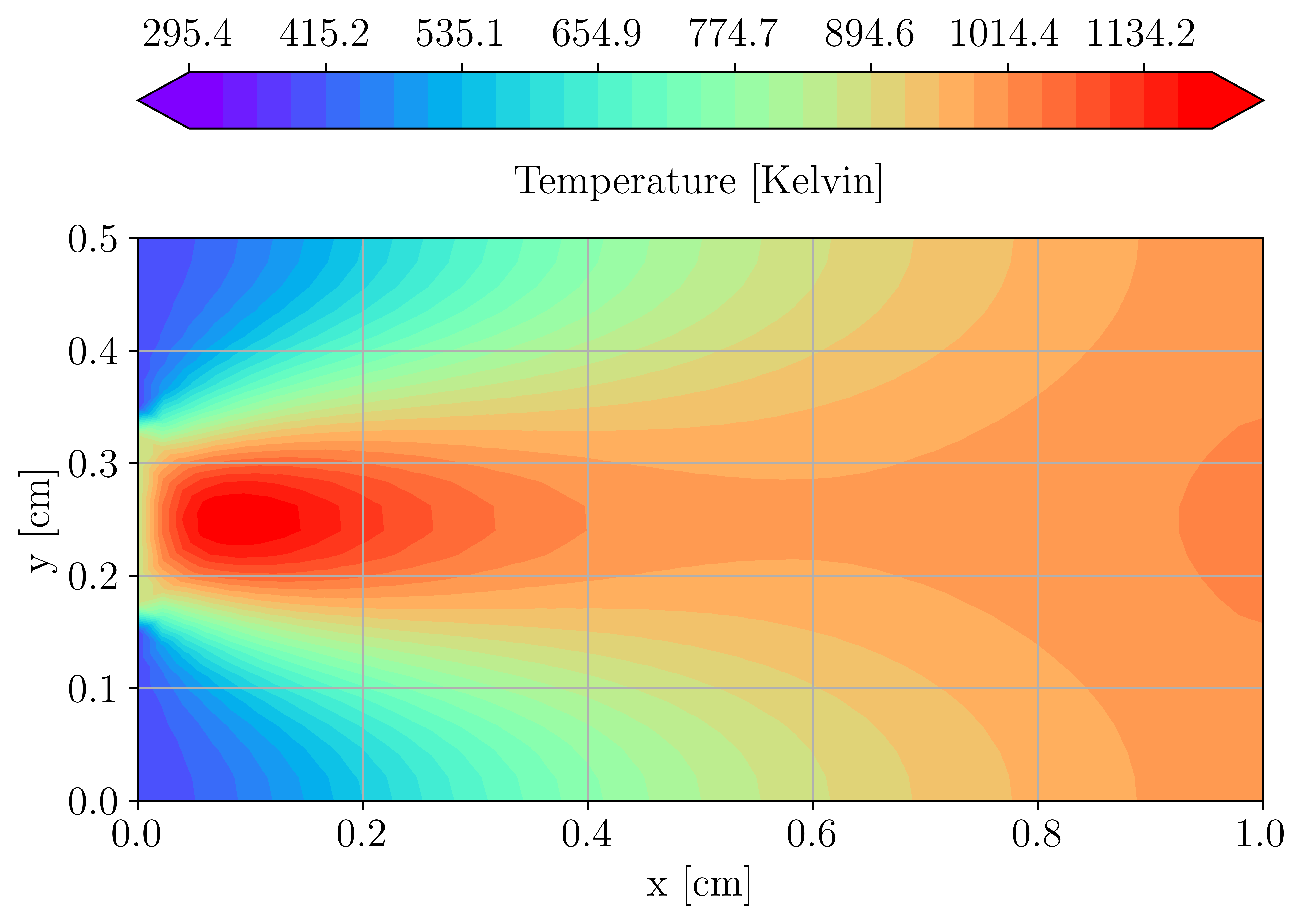}
    \end{subfigure}%
    \hfill
    \begin{subfigure}[b]{\twocolumnlength}%
        \centering
        \includegraphics[width=\linewidth]{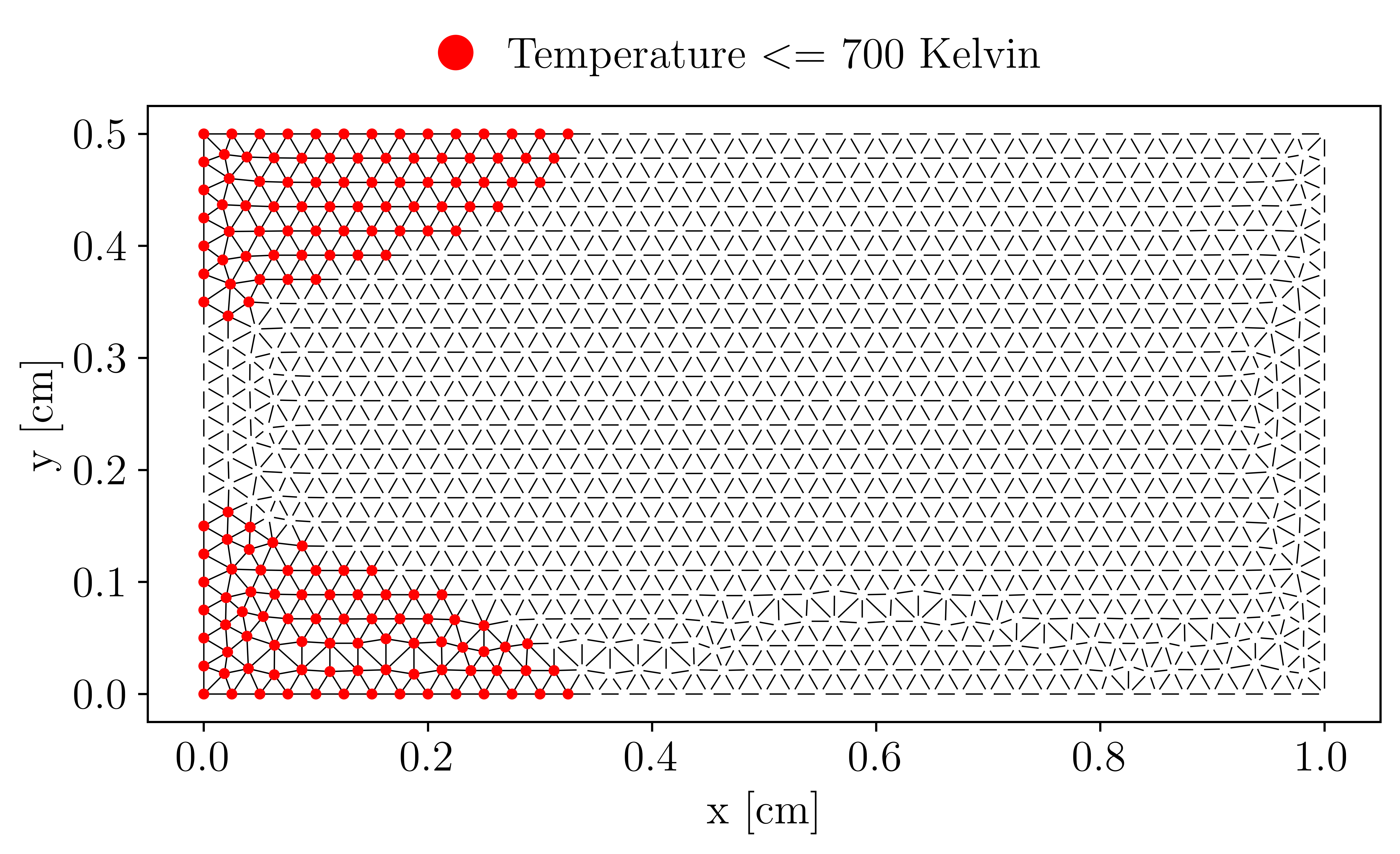}
    \end{subfigure}%
    \caption{Temperature field (left) and associated indicator function (right) for $t_\text{obs}=0.05$ [s], $T_\text{crit}=700$ [K] and the input parameters $(A,E,T_i,T_o,\phi)=(1.4468 \times 10^{12}, 6.4137 \times 10^{3}, 8.5991 \times 10^{2}, 3.9198 \times 10^{2}, 1.3509 \times 10^{0})$. 
    }
    \label{fig:temp_and_indicator_ab}
\end{figure}

\newpage
\subsubsection{Sensitivity indices for different observation windows}
Recall that both the HSIC-ANOVA indices and the SpIn Sobol' indices are defined in terms of volume integrals of the indicator functions of the feasible sets over the spatial domain $\mcX$. {Choosing an observation window $\mcX_\text{obs} \subset \mcX$ effectively restricts the volume integrals to the integration domain $\mcX_\text{obs}$. Thus, both the extent and location of the feasible sets plays an important role for both indices. Further, the wide variability of the feasible sets across the spatial domain which we observed in the previous subsection leads us to anticipate that the sensitivity indices depend on the chosen observation window.}
This dependence is non-trivial: Recall that $S_A^\text{spin}$ in \eqref{SA:spin} is defined as \textit{ratio} of two volume integrals over the spatial domain $\mcX$. Each integrand is non-negative, thus the numerator and denominator are monotonic functions with respect to inclusion of the integration domain. However, the ratio of monotonic functions is not necessarily monotonic. Thus the dependence of the SpIn Sobol' indices on the size and location of the observation window is not easily foreseen.
The situation for the HSIC-ANOVA indices is not straightforward either.
First of all, these indices are again defined as ratios of non-negative quantities, see \eqref{definition:hsic_anova_index}. 
Second, each of the integrands is not merely a volume integral over $\mcX$, see \eqref{equation:hsic_estimation}. Only the set kernel $k_\text{set}$ defined in \eqref{equation:set_kernel} involves a volume integral with a non-negative integrand. 
Notably, the function $k_\text{set}$ is monotonically decreasing when the volume of the integration domain increases.
Overall, it is difficult to make general conclusions about the dependence of the HSIC-ANOVA indices on the observation window.

To estimate the sensitivity indices and to carry out a statistical analysis of the estimates, we perform 10 repeated estimations of the SpIn Sobol' indices using 1,000 data points, and 10 repeated estimations of the HSIC-ANOVA indices using 1,000 and 10,000 data points, respectively. We perform these estimates for different observation windows $\mcX_\text{obs}$ in \eqref{equation:excursion_set_for_cdr_temp_output}.
In Figure~\ref{fig:results_SA_hsic_sobol_1000} we present the results for $\mcX_\text{obs}=[0.0,0.1]\times[0.0,0.5]$ in the left column, for $\mcX_\text{obs}=[0.0,0.3]\times[0.165,0.33]$ in the middle column, and for the full domain $\mcX_\text{obs}=\mcX=[0.0,0.5]\times[0.0,1.0]$ in the right column. 

\begin{figure}[!hptb]
    \centering
    \includegraphics[width=\linewidth]{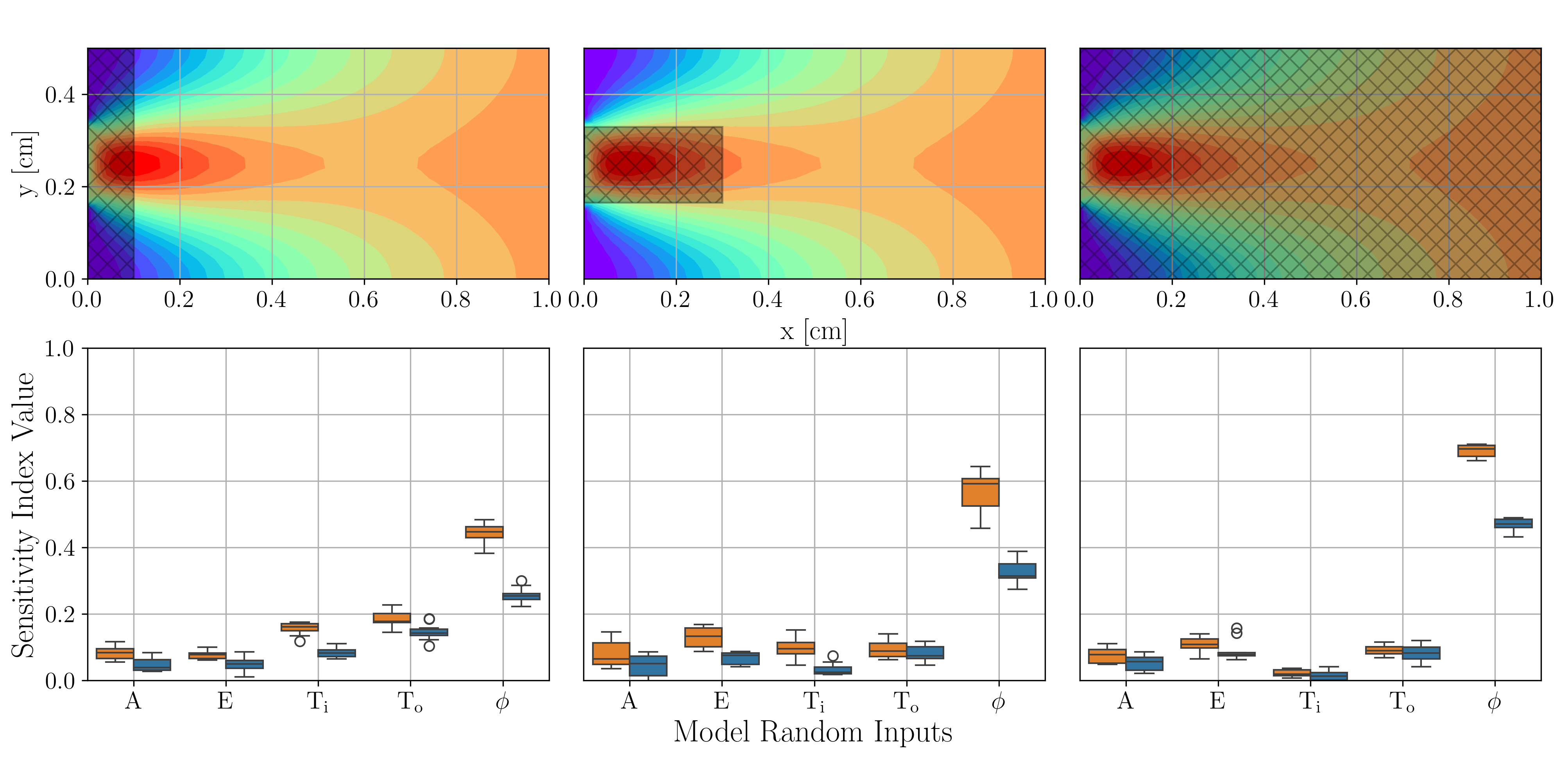}
    \caption{Box plots of estimated spatially-integrated \textbf{\textcolor{tabblue}{Sobol}'} sensitivity indices (blue) and \textbf{\textcolor{taborange}{HSIC-ANOVA}} indices (orange) for the temperature field of the CDR problem of Equation \eqref{eqn:experims_models_cdr_pde}, repeated 10 times for 1000 data points each within the respective sub-domain on the top row. 
    }
    \label{fig:results_SA_hsic_sobol_1000}
\end{figure}

First, we notice that in each observation window the relative importance of the random inputs of the CDR model estimated by both the spatially-integrated Sobol' indices and the HSIC-ANOVA indices is similar. 
For example, for $\mcX_\text{obs}=\mcX$, the fuel-to-oxidizer ratio $\phi$ is the most important, followed by the boundary temperature $T_o$ and the activation energy $E$, and finally we have the pre-exponential factor $A$ and the inflow temperature  $T_i$. 
However, the absolute values of the indices are different, with $\phi$ having an average index value of about $0.7$ for the HSIC-ANOVA, and $0.5$ for the SpIn Sobol' index. The qualitative similarity between the indices points to the understanding that the SpIn Sobol' and HSIC-ANOVA indices capture the same information and input-output relationship in this example.

Next we observe that, depending on where we draw the observation window in the spatial domain, the relative importance of the input parameters is noticeably different. The exception is the fuel-to-oxidizer ratio $\phi$, which is the most important input parameter regardless of the observation window. This is consistent with the fact that $\phi$ determines how much fuel is available at the inlet to burn.
For the other inputs, however, we observe an importance-flipping behavior. 
For example, if we observe the entire spatial domain, the sensitivity indices estimate $T_o$ and $E$ appear to be about equally important, followed by the pre-exponential factor $A$ and the inflow-boundary temperature $T_i$, see Figure~\ref{fig:results_SA_hsic_sobol_1000} on the right.
If the observation window is a small section near the left vertical boundary, then this raises the importance of $T_o$ and $T_i$ as second and third most important, while reducing the importance of the reaction coefficients $E$ and $A$ to 4th-most and least important, see Figure~\ref{fig:results_SA_hsic_sobol_1000} on the left.

The variability in the box plots of Figure~\ref{fig:results_SA_hsic_sobol_1000} makes it difficult to compare the relative importance of inputs in some cases. 
In the middle plot of Figure~\ref{fig:results_SA_hsic_sobol_1000}, the box plot whiskers of all but $\phi$ overlap with the boxes of other inputs. 
To reduce this variability we use more data points for the esimation of the HSIC-ANOVA indices.
As discussed in Section~\ref{section:estimation_of_sensitivity_indices}, the estimation of the first-order SpIn Sobol' indices requires many more evaluations of the process model, hence we focus on the HSIC-ANOVA indices. 
We perform 10 repeats of the HSIC-ANOVA estimation using 10,000 data points each, for the same observation windows as in Figure~\ref{fig:results_SA_hsic_sobol_1000}.  
The results are presented in Figure~\ref{fig:results_SA_hsic_10000}. 
\begin{figure}[!hptb]
    \centering
    \includegraphics[width=\linewidth]{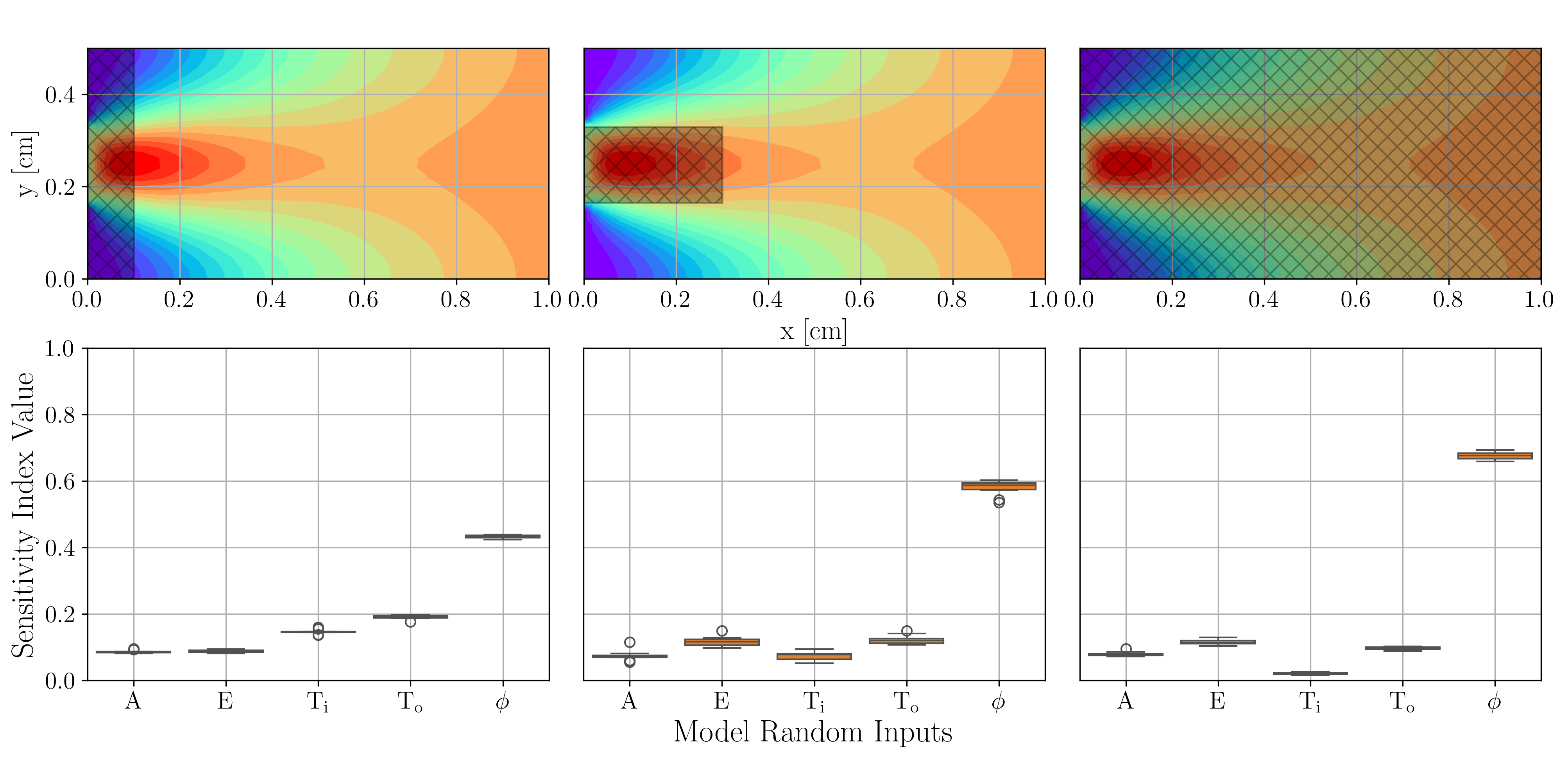}
    \caption{Box plots of estimated \textbf{\textcolor{taborange}{HSIC-ANOVA}} indices for the temperature field of the CDR problem of Equation \eqref{eqn:experims_models_cdr_pde}, repeated 10 times for 10000 data points each within the respective sub-domain on the top row. 
    }
    \label{fig:results_SA_hsic_10000}
\end{figure}
Comparing Figures~\ref{fig:results_SA_hsic_sobol_1000} and \ref{fig:results_SA_hsic_10000} we see that the relative positions in the importance profiles are the same, and the average value in each box plot is similar.
Moreover, the variability of each box plot is reduced when using more data points.
This allows us to compare the estimated importance profiles of the inputs with more confidence. For example, when we study the temperature field's sensitivity in passing the 700 Kelvin threshold in the observation window in the middle plot of Figure~\ref{fig:results_SA_hsic_10000}, the parameters $T_o$ and $E$ are found to be similarly important after the most important input $\phi$, while $T_i$ and $A$ are found to be similar and the least important of all inputs. The left-most plot in Figure~\ref{fig:results_SA_hsic_10000} tells us that in the observation window close to the left vertical boundary, the left-boundary temperatures $T_i$ and $T_o$ are the most important after $\phi$. 
Finally, if we observe the complete spatial domain, $T_i$ is estimated to be the least important, and $E$ and $T_o$ the second most important inputs.

\section{Conclusion}\label{section:conclusion}

In this work we performed a global, statistical SA for process models with set-valued outputs.
We studied two sensitivity measures: the HSIC-ANOVA indices for set-valued outputs \cite{fellmann_kernelSA} which work with kernel mean embeddings of unconditional measures on the input-output space of the model, and generalized, spatially-integrated Sobol' indices \cite{gamboa_SA_for_multiDim_and_funtional_outputs} for function-valued outputs which are associated with (conditional) measures on the output space of the model.
We proposed efficient estimators for both indices for FEM-based numerical models, thereby improving existing estimators in the literature. 
Moreover, we developed a Python-based open-source code for both HSIC-ANOVA and spatially-integrated Sobol' indices which can be coupled with FEM-based numerical models.
The datasets and code associated with the simulations in this work are available in public repositories.

We used the indices to rank the random input parameters of a hydrogen combustion model, where the temperature in the combustion domain remains below a given threshold. 
We found that in this example the HSIC-ANOVA indices and the SpIn Sobol' indices give qualitatively similar rankings.
In particular, the fuel-to-oxidizer ratio is consistently ranked as the most import input variable. 
We also found that the importance ranking of the model inputs changes and in some cases ``flips'' depending on where we place the observation window in the spatial domain.
This tells us that different dimension-reduced order models may be needed depending on the spatial location. 
Since the HSIC-ANOVA indices are relatively cheap to estimate, they might be the preferable choice in practical applications, where the numerical solver is expensive and the computational resources are limited.

There are many topics for further investigations, for example, a comparison with the kernel-based indices by Barr and Rabitz \cite{Barr:2022} or the study of more complex constraints in the definition of the spatial sets which may depend on time or the PDE solution.
The kernel-based sensitivity analysis for correlated inputs is relatively unexplored with some recent ideas presented by Larsen and Alexanderian for first-order and total HSIC indices \cite{2026TotalHSICwithDependentInputs}.
Moreover, the choice of the input and output kernels in the HSIC-ANOVA indices requires further mathematical analysis.

\section*{Acknowledgments}
The authors thank Jules Pertinand for useful comments and discussions. CP acknowledges the support of the TUM International Graduate School of Science and Engineering through the project DUCS.

\section*{Data Availability}
The code used in this work is publicly available at \url{https://github.com/Farbodch/sevosa}. 
The dataset of all simulations used to obtain the results discussed here is available at \url{https://doi.org/10.5281/zenodo.21771890}.

\section*{Declaration of AI use}
{Google's Gemini 3.5 Flash was used to identify sources for the theory on Lebesgue-measurable sets and BV functions, and for proof-checking. The authors assume responsibility for all theoretical and numerical results in this manuscript.} 
The authors used OpenAI's ChatGPT (models GPT-5, GPT-5.1, GPT-5.2, and GPT-5.4) as an auxiliary software-development aid. The tool was used to provide explanations of submitted error messages and to assist in identifying potentially relevant external software packages and understanding the functionality and usage of package methods. In some cases, ChatGPT produced short illustrative code snippets showing how an external package function or method might be called or configured to achieve a specified behaviour. These snippets were used solely as explanatory examples for understanding package interfaces and functionality. They were not copied into, adapted for, or otherwise incorporated into the published codebase. After consulting such examples, the authors independently implemented the required functionality, and all resulting code was written, reviewed, tested, and verified by the authors. The open-source codebase used to generate the data, analyses, and results reported in this paper contains no AI-generated code. 
{The authors assume responsibility for the computational implementation and all results reported in the paper.}

\appendix 

\section{Supplementary Material}

\definecolor{AlgorithmNodeFill}{HTML}{EDF2F7}
\definecolor{AlgorithmNodeText}{HTML}{365A80}
\definecolor{DataFileFill}{HTML}{E8AEA2}

\tikzset{
    spinSob node/.style={
        rectangle,
        minimum width=17mm, 
        minimum height=7mm, 
        inner xsep=2pt, 
        inner ysep=1pt,
        draw=black,
        line width=0.7pt,
        fill=AlgorithmNodeFill,
        text=tabblue,
        font=\small\bfseries,
        align=center
    },
    hsic node/.style={
        rectangle,
        minimum width=17mm, 
        minimum height=7mm, 
        inner xsep=2pt, 
        inner ysep=1pt,
        draw=black,
        line width=0.7pt,
        fill=AlgorithmNodeFill,
        text=taborange,
        font=\small\bfseries,
        align=center
    },
    data file node/.style={ 
        rectangle,
        rounded corners=1.5pt,
        minimum width=17mm, 
        minimum height=7mm, 
        inner xsep=2pt, 
        inner ysep=1pt,
        draw=black, 
        line width=0.7pt, 
        fill=DataFileFill, 
        text=black, 
        font=\small\bfseries, 
        align=center
    },
    algorithm edge/.style={
        draw=black,
        line width=0.7pt,
        -{Stealth[length=1.8mm, width=1.3mm]}
    },
    data edge/.style={
        draw=black,
        dashed,
        dash pattern=on 3pt off 2pt,
        line width=0.7pt,
        -{Stealth[length=1.8mm, width=1.3mm]}
    },
}

\newcommand{\SpInSobAlgorithmMap}{
    \begin{tikzpicture}[baseline=(current bounding box.north)]
        \node[spinSob node] (dataGenSob) at (-1.3, 0)
            {Alg.\\\ref{SM_algorithm:data_generation_sobol}};
        \node[data file node] (dataFileSob) at (-1.3, -1.7)
            {Data File};
        
        \node[spinSob node] (alg1) at (1.0, 0)
            {Alg.\\\ref{SM_algorithm:1_SpInSobol_main}};
        \node[spinSob node] (alg2) at (1.0, -1.7)
            {Alg.\\\ref{SM_algorithm:2_SpInSobol_integrated_spatial_cov}};
            
        \draw[algorithm edge]
            (alg1) -- (alg2);
        \draw[data edge]
            (dataGenSob) -- (dataFileSob);
        \draw[data edge]
            (dataFileSob) -- (alg1);
    \end{tikzpicture}
}

\newcommand{\HSICAlgorithmMap}{
    \begin{tikzpicture}[baseline=(current bounding box.north)]
        \node[hsic node] (dataGenHSIC) at (-2.5, 0)
            {Alg.\\\ref{SM_algorithm:data_generation_hsic}};
        \node[data file node] (dataFileHSIC) at (-2.5, -1.7)
            {Data File};
        
        \node[hsic node] (alg3) at (0.5,0)
            {Alg.\\\ref{SM_algorithm:3_HSIC_main}};
        \node[hsic node] (alg4) at (-0.5, -1.7)
            {Alg.\\\ref{SM_algorithm:4_HSIC_gamma_matrix}};
        \node[hsic node] (alg5) at (-0.5, -3.4)
            {Alg.\\\ref{SM_algorithm:5_HSIC_spatially_integrated_symm_diff}};
        \node[hsic node] (alg6) at (1.5, -1.7)
            {Alg.\\\ref{SM_algorithm:6_HSIC_computing_input_kernel_matrix}};
        \node[hsic node] (alg7) at (1.5, -3.4)
            {Alg.\\\ref{SM_algorithm:7_HSIC_order_1_anova_input_kernel}};
        
        \draw[algorithm edge]
            (alg3) -- (alg4);
        \draw[algorithm edge]
            (alg3) -- (alg6);
        \draw[algorithm edge]
            (alg4) -- (alg5);
        \draw[algorithm edge]
            (alg6) -- (alg7);

        \draw[data edge]
            (dataGenHSIC) -- (dataFileHSIC);
        \draw[data edge]
            (dataFileHSIC) -- (alg3);
    \end{tikzpicture}
}

The supplementary materials contain the algorithms and routines for estimating the spatially-integrated (SpIn) Sobol' sensitivity indices (Section~\ref{sm:sec:spin}),
the HSIC-ANOVA indices (Section~\ref{sm:sec:hsic}), and the data generation (Section~\ref{sm:sec:data}).
The specific configurations for the CDR test problem are recorded in Section~\ref{sec:CDR specific}.
The dependency map of the algorithms is shown in Figure \ref{SM_fig:algorithm_dependency_maps}.

\begin{figure}[htp]
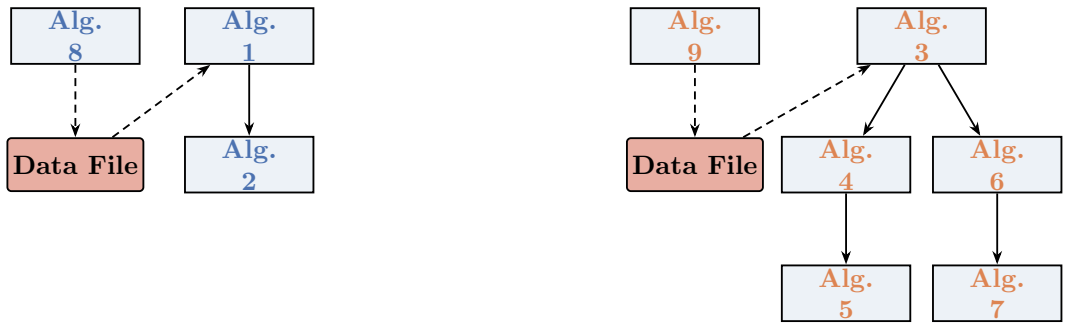

    \centering
    
    \begin{subfigure}[t]{0.32\textwidth}
        \vspace{0pt}
        \centering
        \SpInSobAlgorithmMap
    \end{subfigure} %
    \hfill
    \begin{subfigure}[t]{0.64\textwidth}
        \vspace{0pt}
        \centering
        \HSICAlgorithmMap
    \end{subfigure}
    
    \caption{Dependency map of the algorithms for the data generation and estimation of the \textcolor{tabblue}{spatially-integrated Sobol'} index (left column) and the  \textcolor{taborange}{HSIC-ANOVA} index (right column). The dashed arrows point from the source to the target. 
    Solid arrows point from the caller function to the called function.}
    \label{SM_fig:algorithm_dependency_maps}
\end{figure}

\subsection{Spatially-integrated Sobol' indices}\label{sm:sec:spin} 
The estimation is detailed in Algorithms \ref{SM_algorithm:1_SpInSobol_main} and \ref{SM_algorithm:2_SpInSobol_integrated_spatial_cov}.

\begin{algorithmenv} \label{SM_algorithm:1_SpInSobol_main}
    \caption{SpIn Sobol' Sensitivity Index}
    \begin{algorithmic}[1]
        \Require number of samples $N$, 
        \Statex \hspace{\algorithmicindent}
            number of random inputs $d$,
        \Statex \hspace{\algorithmicindent}
            \texttt{data-directory},
            \texttt{mesh-directory},
        \Statex \hspace{\algorithmicindent}
            \texttt{g-constraint},
            \texttt{test-domain},
        \Statex \hspace{\algorithmicindent}
            \texttt{index-order-numbers},
            \texttt{total-sobols-flag}.
        
        \Statex
        \State \texttt{index-set} $\gets$ generate $\{A \in \mathcal{P}_d : |A| \in \texttt{index-order-numbers}\}$
        \If{\texttt{total-Sobols-flag} is True}
            \State \texttt{index-set} $\gets$ generate $\{A \in \mathcal{P}_d : |A| = d-1\}$
        \EndIf
        
        \Statex
        \State $X_h \gets$ load FEM mesh from \texttt{mesh-directory}
        \State $X_h^t \gets$ $\{\text{cell } C \in X_h : \text{all vertices of } C \text{ lie in \texttt{test-domain}}\}$    
        \State $\texttt{M} \gets$ assemble mass matrix from $X_h^t$
        
        \Statex
        \AlgAssign{\texttt{g-I-directories}}{
            load $N$ directories for FEM
            solutions $g_I$ from \texttt{data-directory}
        }
            
        \For{$A \in \texttt{index-set}$}
            \AlgAssign{\texttt{data-directory-A}}{set the appropriate sub-folder directory to index $A$ data inside \texttt{data-directory}}
            \AlgAssign{ \texttt{g-A-II-directories}}{
                load $N$ directories for FEM solutions $g_{A,II}$ from \texttt{data-directory-A}
            }
            
            \AlgAssign{\texttt{g-A-tilde-directories}}{
                load $N$ directories for FEM solutions $g_{A,\sim}$ from \texttt{data-directory-A}
            }
            
        \EndFor

        \Statex
        \For{\texttt{g-I-directory} in \texttt{g-I-directories}}
            \State load $g_I(x)$ from \texttt{g-I-directory} to memory
            \AlgAssign{\texttt{indicators-I-list}}{
                compute indicator function on $g_I(x)$, $I_I(x) \coloneqq \mathds{1}_{\{g_I(x) \leq \texttt{g-constraint}\}}$
            }
        \EndFor
        
        \Statex
        \AlgAssign{\texttt{cov-I}}{
            call \texttt{integrated-spatial-cov} and pass in 
            (\texttt{indicators-I-list}, \texttt{indicators-I-list}, \texttt{M})    
        }
        \For{$A \in \texttt{index-set}$}
            \For{\texttt{g-A-II-directory} in \texttt{g-A-II-directories}}
                \State load $g_{A,II}(x)$ from \texttt{g-A-II-directory}
                \AlgAssign{\texttt{indicators-A-II-list}}{compute indicator function on $g_{A,II}(x)$, $I_{A,II}(x) \coloneqq \mathds{1}_{\{g_{A,II}(x) \leq g\text{-constraint}\}}$}
            \EndFor
            \For{\texttt{g-A-tilde-directory} in \texttt{g-A-tilde-directories}}
                \State load $g_{A,\sim}(x)$ from \texttt{g-A-tilde-directory}
                \AlgAssign{\texttt{indicators-A-tilde-list}}{compute indicator function on $g_{A,\sim}(x)$, $I_{A,\sim}(x) \coloneqq \mathds{1}_{\{g_{A,\sim}(x) \leq g\text{-constraint}\}}$}
            \EndFor
            \AlgAssign{\texttt{cov-A-list}}{
                call \texttt{integrated-spatial-cov} and pass in
                (\texttt{indicators-A-II-list}, \texttt{indicators-A-tilde-list}, \texttt{M})
            }
        \EndFor
        
        \Statex
        \For{$A \in \texttt{index-set}$}
            \State $\texttt{cov-A}$ $\gets$ \texttt{cov-A-list}
            \If{\texttt{total-sobols-flag} is True \textbf{and} $|A| = d-1$}
                \State
                \begin{equation*}
                    S_{A}^{\mathrm{total}} = 1 - \frac{\texttt{Cov-A}}{\texttt{Cov-I}}
                \end{equation*}
                \State \texttt{index-results} $\gets$ $S_{A}^{\mathrm{total}}$
            \Else
                \State
                \begin{equation*}
                    S_A = \frac{\texttt{Cov-A}}{\texttt{Cov-I}}
                \end{equation*}
                \State \texttt{index-results} $\gets$ $S_A$
            \EndIf
        \EndFor
        \State \Return \texttt{index-results}
    \end{algorithmic}
\end{algorithmenv}

\begin{algorithmenv} \label{SM_algorithm:2_SpInSobol_integrated_spatial_cov}
    \caption{\texttt{integrated-spatial-cov}}
    \begin{algorithmic}[1]
        \Require \texttt{FEM-functions-list-I}, \texttt{ FEM-functions-list-II}, mass matrix\texttt{ M}.

        \Statex
        \State $N \gets$ \texttt{sampling-number} (from length of \texttt{FEM-functions-list-I})
        \State \texttt{FEM-{\bf c}-I-list} $\gets$ coefficients in \texttt{FEM-functions-list-I}
        \State \texttt{FEM-{\bf c}-II-list} $\gets$ coefficients in \texttt{FEM-functions-list-II}
        \State $\bar{\bm{c}}_I = \frac{1}{N}\sum_{c_I\in\texttt{FEM-{\bf c}-I-list}}c_I$
        \State $\bar{\bm{c}}_{II} = \frac{1}{N}\sum_{c_{II}\in\texttt{FEM-{\bf c}-II-list}}c_{II}$
        \State Compute
            \begin{equation*}
                \texttt{spatial-cov} \gets \frac{1}{N}\sum_{i=1}^N{(c_I^{(i)})}^T\texttt{M}c_{II}^{(i)} - {(\bar{\bm{c}}_{I})}^T\texttt{M}\bar{\bm{c}}_{II}, \quad \bm{c}_{I/II}^{(i)}\in \texttt{FEM-{\bf c}-I/II-list}.
            \end{equation*}
        \State \Return \texttt{spatial-cov}
    \end{algorithmic}
\end{algorithmenv}

\subsection{HSIC-ANOVA indices}\label{sm:sec:hsic} 

The FEM optimized estimation of the HSIC-ANOVA indices is detailed in Algorithms \ref{SM_algorithm:3_HSIC_main}, \ref{SM_algorithm:4_HSIC_gamma_matrix}, \ref{SM_algorithm:5_HSIC_spatially_integrated_symm_diff}, \ref{SM_algorithm:6_HSIC_computing_input_kernel_matrix}, and \ref{SM_algorithm:7_HSIC_order_1_anova_input_kernel}. 

\begin{algorithmenv} \label{SM_algorithm:3_HSIC_main}
    \caption{FEM-optimized Estimation of the HSIC Sensitivity Index}
    \begin{algorithmic}[1]
        \Require \text{sampling number $N$},
            \texttt{mesh-directory},
        \Statex \hspace{\algorithmicindent}
            \texttt{data-directory}, 
            \texttt{input-domain-specifications},
        \Statex \hspace{\algorithmicindent}
            \texttt{g-constraint},
            \texttt{test-domain}.
        
        \Statex
        \State $d \gets$ number of random inputs (from \texttt{input-domain-specifications})  
        \State \texttt{index-set} $\gets$ generate $\{A \in \mathcal{P}_d : |A|=1 \text{ or } |A|=d\}$

        \Statex
        \State $X_h \gets$ load FEM mesh from \texttt{mesh-directory}
        \State $X_h^t \gets$ $\{\text{cell } C \in X_h : \text{all vertices of } C \text{ lie in \texttt{test-domain}}\}$    
        \State $\texttt{M} \gets$ assemble mass matrix from $X_h^t$
        
        \Statex
        \AlgAssign{\texttt{g-directories}}{load $N$ directories for FEM function solutions $g$ from \texttt{data-directory}}
        
        \AlgAssign{\texttt{u-input-data}}{
            load $N$ input data $\bm{u}$ ($\in\mathbb{R}^d$) from \texttt{data-directory} such that each $\bm{u}_i$ inside \texttt{u-input-data} corresponds to solution $\bm{g}_i$ inside \texttt{g-directories} for all $i\in\{1,\dots,N\}$
        }
        \Statex
        \For{\texttt{g-directory} in \texttt{g-directories}}
            \State load $\bm{g}(\bm{x})$ from \texttt{g-directory} to memory
            \AlgAssign{\texttt{indicators-list}}{compute indicator function on $\bm{g}(\bm{x})$, $I(\bm{x}) \coloneqq \mathds{1}_{\{\bm{g}(\bm{x}) \leq \texttt{g-constraint}\}}$}
        \EndFor
        
        \Statex
        \State $K_{\Gamma} \gets$ call \texttt{compute-gamma-matrix(indicators-list, M)}

        \Statex
        \AlgAssign{\texttt{u-input-data-transformed}}{
            transform input data $\bm{u}$ from the distribution it was sampled from to $\text{Uniform}([0,1])$, according to \texttt{input-domain-specifications} for each $\bm{u}$ in \texttt{u-input-data}
        }

        \Statex
        \For{index $A\in\texttt{index-set}$}
            \AlgAssign{$K_A$}{
                call \texttt{compute-input-kernel-matrix} and pass in
                ($A$, \texttt{u-input-data-transformed})
            } 
        \EndFor

        \Statex
        \For{index $A\in\texttt{index-set}$}
            \State 
                \begin{equation*}
                    {\texttt{HSIC}}_A \gets \frac{2}{N(N-1)}\sum_{j=2}^N\sum_{i=1}^{j-1}K^{(i,j)}_AK^{(i,j)}_{\Gamma},\quad K^{(i,j)}_A\in K_A, K^{(i,j)}_{\Gamma}\in K_{\Gamma}
                \end{equation*}
        \EndFor

        \Statex
        \State $\texttt{HSIC} \gets \texttt{HSIC}_A$ for $|A|=d$
        \For{index $A\in\{A\in\mathcal{P}_d:|A|=1\}$}
            \State 
                \begin{equation*}
                    S^{\texttt{HSIC}}_A \gets \frac{\texttt{HSIC}_A}{\texttt{HSIC}}
                \end{equation*}
            \State \texttt{index-results} $\gets$ $S^{\texttt{HSIC}}_A$
        \EndFor
        \State \Return \texttt{index-results}
    \end{algorithmic}
\end{algorithmenv}

\begin{algorithmenv} \label{SM_algorithm:4_HSIC_gamma_matrix}
    \caption{\texttt{compute-gamma-matrix}}
    \begin{algorithmic}[1]
        \Require \texttt{FEM-functions-list}, mass matrix \texttt{M}.

        \Statex
        \State $N \gets$ \texttt{sampling-number} (from length of \texttt{FEM-functions-list})
        \AlgAssign{$\Lambda$}{
            call \texttt{spatial-integrated-symmetric-difference} and pass in
            (\texttt{FEM-functions-list}, \texttt{M})
        }
        \State $\sigma^2\gets \frac{1}{N^2}\sum_{i,j=1}^N\Lambda^{(i,j)}$

        \For{$i,j\in{1,2,\dots,N}$}
            \If{$i\neq j$}
                \State $K_{\Gamma}^{(i,j)} \gets$ $\exp(-\frac{\Lambda^{(i,j)}}{2\sigma^2})$
            \Else
                \State $K_{\Gamma}^{(i,j)} \gets 0$ 
            \EndIf
        \EndFor
        \State \Return $K_{\Gamma}$
    \end{algorithmic}
\end{algorithmenv}

\begin{algorithmenv} \label{SM_algorithm:5_HSIC_spatially_integrated_symm_diff}
    \caption{\texttt{spatial-integrated-symmetric-difference}}
    \begin{algorithmic}[1]
        \Require \texttt{FEM-functions-list}, mass matrix \texttt{M}.

        \Statex
        \State $N \gets$ \texttt{sampling-number} (from length of \texttt{FEM-functions-list})
        \State \texttt{FEM-{\bf c}-list} $\gets$ FEM coefficient vectors of functions in \texttt{FEM-functions-list}
        \For{$i,j\in\{1,2,\dots,N\}$}
            \State 
            \begin{equation*}
                \Lambda^{(i,j)} \gets {(\bm{c}^{(i)}-\bm{c}^{(j)})}^T\texttt{M}{(\bm{c}^{(i)}-\bm{c}^{(j)})} \quad \bm{c}^{(i)},\bm{c}^{(j)}\in \texttt{FEM-{\bf c}-list}.
            \end{equation*}
        \EndFor 
        \State \Return matrix $\Lambda$
    \end{algorithmic}
\end{algorithmenv}

\begin{algorithmenv} \label{SM_algorithm:6_HSIC_computing_input_kernel_matrix}
    \caption{\texttt{compute-input-kernel-matrix}}
    \begin{algorithmic}[1]
        \Require index $A$, \texttt{u-input-data-transformed} ($\in \mathbb{R}^{N\times d}$).
        
        \Statex
        \For{$i,j\in\{1,2,\dots,N\}$}
            \If{$i < j$ and $i\neq j$}
                \State
                    \begin{equation*}
                        K_A^{(i,j)} \gets [\prod_{l\in A}\texttt{k-ANOVA}(u^{(i)}_{l},u^{(j)}_{l})]-1, \quad u^{(i)}\in\mathbb{R}^d,
                    \end{equation*} 
                    where $l$ is the $l$-th element of $u^{(i)}$
            \Else
                \begin{equation*}
                        K_A^{(i,j)} \gets 0
                \end{equation*} 
            \EndIf
        \EndFor
        \State \Return $K_A$
    \end{algorithmic}
\end{algorithmenv}

\begin{algorithmenv} \label{SM_algorithm:7_HSIC_order_1_anova_input_kernel}
    \caption{\texttt{k-ANOVA}}
    \begin{algorithmic}[1]
        \Require $a\in[0,1]$, $b\in[0,1]$.
        
        \Statex
        \State $k \gets 1 + (a-\frac{1}{2})(b-\frac{1}{2}) + \frac{1}{2}[{(a-b)}^2 - |a-b| + \frac{1}{6}]$
        \State \Return k
    \end{algorithmic}
\end{algorithmenv}

\subsection{Data generation}\label{sm:sec:data}
The data generation routine for the SpIn Sobol' indices in Algorithm \ref{SM_algorithm:data_generation_sobol} is tailored towards a pick-freeze estimator \cite{gamboa_pickFreeze, janon2014asymptotic_of_aggr_and_scal_PickFreeze}  and assumes suitable data generation processes and storage structures. 
The data generation routine for the HSIC-ANOVA indices in Algorithm \ref{SM_algorithm:data_generation_hsic} makes no such assumption.

\begin{algorithmenv} \label{SM_algorithm:data_generation_sobol}
    \caption{Data generation for SpIn Sobol' indices and FEM-based process model $\bm{g}(\bm{x})=G(\bm{u},\bm{x}, t)$, time $t\in\mathbb{R}_{>0}$, spatial parameter $\bm{x}$, random input vector $\bm{U}$, realization $\bm{u}\in\mathbb{R}^d$, and $\bm{g}(\bm{x})\in\mathbb{R}^p$.}
    \begin{algorithmic}[1]
        \Require distribution $\mathbb{P}_i$ for each random input $U_i$, $i\in\{1,\dots,d\}$, 
        \Statex \hspace{\algorithmicindent}
            \texttt{index-sets-to-calculate},
         \Statex \hspace{\algorithmicindent}
            \texttt{FEM-parameters}, 
            simulation final time $t_f$, 
        \Statex \hspace{\algorithmicindent}
            \texttt{save-directory}, 
            \texttt{mesh-directory}.
        
        \Statex
        \State $\bm{u}_I \gets$ realization of $\bm{U}$
        \AlgAssign{$\bm{g}_I(\bm{x})$}{
            solve FEM model $G(\bm{u}_I,t_f,\bm{x})$ according to \texttt{FEM-parameters} on mesh loaded from \texttt{mesh-directory}
        }
        \State $\texttt{save-directory} \gets \bm{u}_I$ input data
        \State $\texttt{save-directory} \gets \bm{g}_I(\bm{x})$ FEM solution coefficients
        \For{$A\in\texttt{index-sets-to-calculate}$}
            \State $\bm{u}_{II} \gets$ realization of $\bm{U}$.
            \State $\bm{u}_{III} \gets$ realization of $\bm{U}$.
            \State $\bm{u}_{\sim}\coloneqq(\bm{u}_{II_A,}\bm{u}_{III_{A^c}})$
            \AlgAssign{$\bm{g}_{A,II}(\bm{x})$}{
                solve FEM model $G(\bm{u}_{II},t_f,\bm{x})$ according to \texttt{FEM-parameters} on mesh loaded from \texttt{mesh-directory}
            }
            \AlgAssign{$\bm{g}_{A,\sim}(\bm{x})$}{
                solve FEM model $G(\bm{u}_{\sim},t_f,\bm{x})$ according to \texttt{FEM-parameters} on mesh loaded from \texttt{mesh-directory}
            }
            \State $\texttt{save-directory} \gets \bm{u}_{II}$ input-data
            \State $\texttt{save-directory} \gets \bm{u}_\sim$ input-data
            \State \texttt{save-directory/A} $\gets$ create new sub folder using string format of index $A$
            \State $\texttt{save-directory/A} \gets \bm{g}_{A,II}(\bm{x})$ FEM solution coefficients
            \State $\texttt{save-directory/A} \gets \bm{g}_{A,\sim}(\bm{x})$ FEM solution coefficients
        \EndFor
    \end{algorithmic}
\end{algorithmenv}

\begin{algorithmenv} \label{SM_algorithm:data_generation_hsic}
    \caption{Data generation for HSIC-ANOVA sensitivity indices and FEM-based process model $\bm{g}(\bm{x})=G(\bm{u},\bm{x}, t)$, time $t\in\mathbb{R}_{>0}$, spatial parameter $\bm{x}$, random input vector $\bm{U}$, realization $\bm{u}\in\mathbb{R}^d$, and $\bm{g}(\bm{x})\in\mathbb{R}^p$.}
    \begin{algorithmic}[1]
        \Require distribution $\mathbb{P}_i$ for each random input $U_i$, $i\in\{1,\dots,d\}$, 
        \Statex \hspace{\algorithmicindent} 
            \texttt{FEM-parameters}, 
            simulation final time $t_f$,
        \Statex \hspace{\algorithmicindent}
            \texttt{save-directory},
            \texttt{mesh-directory}.
        
        \Statex
        \State $\bm{u} \gets$ realization of $\bm{U}$
        \AlgAssign{$\bm{g}(\bm{x})$}{
            solve FEM model $G(\bm{u},t_f,\bm{x})$ according to \texttt{FEM-parameters} on mesh loaded from \texttt{mesh-directory}
        }
        \State $\texttt{save-directory} \gets \bm{u}$ input data
        \State $\texttt{save-directory} \gets \bm{g}(\bm{x})$ FEM solution coefficients
    \end{algorithmic}
\end{algorithmenv}
\bigskip

\subsection{Specific configurations for the hydrogen combustion problem}\label{sec:CDR specific}
We describe the simulation pipeline for the HSIC-ANOVA and the SpIn sensitivity indices by example of the convection-diffusion-reaction (CDR) test problem in Section~\ref{subsection:experims_model_systems_cdr}.
To generate the data for the estimators we sample the random inputs and for each input solve the PDE in the process model numerically. 
We store the resulting FEM approximation of just the temperature solution field. 
We save the results to memory, and then we pass the data directory as input to the respective estimation algorithms together with appropriate configurations to estimate the sensitivity index of interest. 
Note that the SpIn estimation Algorithm~\ref{SM_algorithm:1_SpInSobol_main} does not need the information of which random input realizations resulted in which output values. 
It does, however, assume that the data was generated via a pick-freeze scheme \cite{gamboa_pickFreeze, janon2014asymptotic_of_aggr_and_scal_PickFreeze} and stored with a specific structure, as outlined in Algorithm~\ref{SM_algorithm:data_generation_sobol}. 
The HSIC estimation Algorithm~\ref{SM_algorithm:1_SpInSobol_main} does not make such a data-generation-structure assumption.
However, it requires that each input-output data pair is stored together in a single file, representing one ``data point''.

The estimation of both the SpIn and HSIC-ANOVA indices was done as follows: The data production algorithm was called $n_\text{data}$ times to generate one data file in each call.
Each data file contains $s_\text{data}$ input-output pairs of the CDR model, where the inputs are correlated.
For the HSIC-ANOVA indices it holds $s_\text{data}=1$, whereas for the SpIn indices $s_\text{data}=21$, and the $21$ input samples in each data file are correlated according to the pick-freeze scheme.
Having the data files at hand, we select $n_\text{smpl}$ data files and each data file goes into the estimation of the respective sensitivity index.
We repeat this $n_\text{box}$ times to perform a statistical analysis of the estimated indices.
The relevant configurations are summarized in Table~\ref{table:numerical_estimation_parameters}. Note that all of the generated data was packaged and stored inside HDF5-formatted \cite{HDF5} files using the Python package h5py \cite{h5py}, where a JSON-formatted \cite{json} data ledger keeps track of the generated data inside each HDF5-formatted file, to speed-up read-in from storage during index calculation routines.

\paragraph{Spatially-integrated Sobol' indices} To generate the data to estimate the main and total effect SpIn indices, we call Algorithm~\ref{SM_algorithm:data_generation_sobol} 12,000 times, producing 12,000 data points as input for Algorithm~\ref{SM_algorithm:1_SpInSobol_main}.
Each call of Algorithm~\ref{SM_algorithm:data_generation_sobol} generates and stores to memory $(1+2\cdot \text{cardinality(\texttt{index-sets-to-calculate})})$-many input-output pairs of the form (model parameters, PDE solution). 
One input-output pair is used to calculate a realization of the indicator random field $Z_I^{(i)}$ in \eqref{eq:ZIZII}.
The remaining input-output pairs are used to calculate realizations of $Z_{II}^{(i)}$ in \eqref{eq:ZIZII} and $Z_{\sim A}^{(i)}$ in \eqref{eq:ZnotA} (hence the prefactor 2) for each index set $A$ which is required.
For example, for $d=5$ input parameters as in the CDR problem, we have $5$ main effect indices (i.e. index sets $A \in \mcP_5$ with $|A|=1$) and $5$ total effect indices, where each total effect is associated with the index set $A^c$ and $|A^c|=d-1$.
Thus, the cardinality of the index sets $\{A\in\mathcal{P}_d:|A|=1 \text{ or } |A|=d-1\}$ is $5+5=10$, and each call of Algorithm~\ref{SM_algorithm:data_generation_sobol} generates $1+2 \times 10 = 21$ realizations of the PDE solution. 
Note that in Figure~\ref{fig:results_SA_hsic_sobol_1000} we use a total of $n_\text{smpl} \times n_\text{box}= 10^3 \times 10 = 10^4$ data files, selected consecutively from the $12,000$ data files in $10$ batches of size $10^3$ each, to generate the box plots.
Also note that we only use $1+10=11$ input-output pairs in each data file since we only estimate the main effect indices $S_A^\text{spin}$ in \eqref{SA:spin} where $|A|=1$.

\paragraph{Kernel-based indices}
To generate the data for the estimation of the HSIC-ANOVA indices we call Algorithm~\ref{SM_algorithm:data_generation_hsic} $110,000$ times to generate $110,000$ input-output pairs of the CDR model, where each pair is stored in a separate data file, and where the inputs among pairs are statistically independent.
We split the $110,000$ data files and use $10,000$ data files to generate the box plots in Figure~\ref{fig:results_SA_hsic_sobol_1000}, and $100,000$ data files to generate the box plots in Figure~\ref{fig:results_SA_hsic_10000}.
For Figure~\ref{fig:results_SA_hsic_sobol_1000} we use a total of $10^3 \times 10 = 10^4$ data files, selected consecutively from the $10,000$ data files in $10$ batches of size $10^3$ each.
For Figure~\ref{fig:results_SA_hsic_10000} we use a total of $10^4 \times 10 = 10^5$ data files, selected consecutively from the $100,000$ data files in $10$ batches of size $10^4$ each.

We compute the FEM approximation of all four solution components of the CDR model in Section~\ref{subsection:experims_model_systems_cdr} using the FEniCS software (version 2019.1.0) \cite{fenics}, but we limit our study to the temperature field, and accordingly store the FEM solution data of only this field. The indicator function over the temperature field, $\mathds{1}_{\{Y_T(t=t_\text{obs},\xx,\uu)\leq T_\text{crit}\}}$, is associated with the feasible set. 
We use the FEM with CG1 elements for the numerical solution and approximate the indicator function in the FEM basis. This means that the approximate indicator function is equal to one on mesh vertices if the temperature constraint is satisfied, otherwise the approximate indicator function is equal to zero. 
This procedure is captured in Lines 16, 22, and 26 of Algorithm~\ref{SM_algorithm:1_SpInSobol_main} for the SpIn Sobol' indices and in Line 10 of Algorithm~\ref{SM_algorithm:3_HSIC_main} for the HSIC-ANOVA indices. 
Each experiment in Table~\ref{table:numerical_estimation_parameters} is performed in different \textit{observation windows} of the spatial domain, the so-called \texttt{test-domain} in the code accompanying this work. 
The relevant numerical procedure is captured in Line 7 of Algorithm~\ref{SM_algorithm:1_SpInSobol_main} for the SpIn Sobol' indices and in Line 4 of Algorithm~\ref{SM_algorithm:3_HSIC_main} for the HSIC-ANOVA indices. 
\begin{table}
    \centering
    \begin{tabular}{c|c||c|c||c|c}
        Index & Algorithm & $n_\text{data}$ & $s_\text{data}$&$n_\text{smpl}$ & $n_\text{box}$ \\
        \hline
        SpIn & \ref{SM_algorithm:1_SpInSobol_main} & $12,000$ & $21$ & 1000 & 10 \\
        HSIC &  \ref{SM_algorithm:3_HSIC_main} & $110,000$ & $1$& $\{1000, 10000\}$ & $10$ \\
    \end{tabular}
    \caption{Configuration summary of the numerical experiments with the CDR problem. The constraint is placed on the temperature field and reads $Y_T \leq 700$ K. $n_\text{data}$ is number of data files, $s_\text{data}$ is the number of (correlated) input samples in each data file, $n_\text{smpl}$ is the number of samples to estimate one SpIn or HSIC-ANOVA index, and $n_\text{box}$ is the number of repeated estimates to generate the box plots in Figure~\ref{fig:results_SA_hsic_sobol_1000} and Figure~\ref{fig:results_SA_hsic_10000}.}
    \label{table:numerical_estimation_parameters}
\end{table}
\bigskip

\bibliographystyle{abbrvurl}
\bibliography{literature}

\end{document}